\documentclass[onefignum,onetabnum]{siamart190516} 

\usepackage{lipsum}
\usepackage{amssymb, amsmath, amsfonts, listings, mathrsfs}
\usepackage{graphicx}
\usepackage{epstopdf}
\usepackage{xcolor}
\usepackage[noend]{algpseudocode}
\usepackage{hyperref}

\usepackage[labelfont=bf]{caption}
\ifpdf
\DeclareGraphicsExtensions{.eps,.pdf,.png,.jpg}
\else
\DeclareGraphicsExtensions{.eps}
\fi

\newcommand{\Def}{\stackrel{\mathrm{def}}{=}}

\newcommand{\dom}{{\rm dom \,}}
\newcommand{\beq}{\begin{equation}}
\newcommand{\eeq}{\end{equation}}

\newcommand{\R}{\mathbb{R}}

\newcommand{\E}{\mathbb{E}}

\newcommand{\SetEQ}{\setcounter{equation}{0}}

\newcommand{\refGE}[1]{\ensuremath{\stackrel{\cref{#1}}{\geq}}}

\renewcommand\arraystretch{1}
\newcommand{\ba}{\begin{array}}
	\newcommand{\ea}{\end{array}}
\newcommand{\beann}{\begin{eqnarray*}}
	\newcommand{\eeann}{\end{eqnarray*}}
\newcommand{\bea}{\begin{eqnarray}}
\newcommand{\eea}{\end{eqnarray}}

\newcommand{\BT}{\begin{theorem}}
	\newcommand{\ET}{\end{theorem}}
\newcommand{\BL}{\begin{lemma}}
	\newcommand{\EL}{\end{lemma}}
\newcommand{\BC}{\begin{corollary}}
	\newcommand{\EC}{\end{corollary}}
\newcommand{\BE}{\begin{example}}
	\newcommand{\EE}{\end{example}}
\newcommand{\BD}{\begin{definition}}
	\newcommand{\ED}{\end{definition}}
\newcommand{\BR}{\begin{remark}}
	\newcommand{\ER}{\end{remark}}
\newcommand{\BAS}{\begin{assumption}}
	\newcommand{\EAS}{\end{assumption}}
\newcommand{\BI}{\begin{itemize}}
	\newcommand{\EI}{\end{itemize}}

\newcommand{\BMP}{\begin{minipage}{9.5cm}}
	\newcommand{\EMP}{\end{minipage}}
\newcommand{\MPT}{\begin{minipage}{11.5cm}}
	\newcommand{\EPT}{\end{minipage}}

\newcommand{\la}{\langle}
\newcommand{\ra}{\rangle}

\def\argmin{\mathop{\rm argmin}}

\newsiamremark{remark}{Remark}
\newsiamremark{hypothesis}{Hypothesis}
\crefname{hypothesis}{Hypothesis}{Hypotheses}
\newsiamthm{claim}{Claim}
\newsiamthm{example}{Example}
\newsiamthm{assumption}{Assumption}

\headers{Contracting Proximal Methods}{Nikita Doikov and Yurii Nesterov}

\slugger{siopt}{}{30}{4}{3146-3169}

\title{Contracting Proximal Methods for
	Smooth Convex \\ Optimization\thanks{Received by the editors December 18, 2019; accepted for publication (in revised form) August
		5, 2020; published electronically November 10, 2020. 
		\url{https://doi.org/10.1137/19M130769X}
		\funding{The research results of this paper were
			obtained in the framework of ERC Advanced Grant 788368.}}}

\author{Nikita Doikov\thanks{Institute of Information and
		Communication Technologies, Electronics and Applied
		Mathematics (ICTEAM), Catholic University of Louvain (UCL),
		1348 Louvain-la-Neuve, Belgium
		(\email{Nikita.Doikov@uclouvain.be}).
	    ORCID: 0000-0003-1141-1625.}
	\and Yurii Nesterov\thanks{Center for Operations Research and
		Econometrics (CORE), Catholic University of Louvain (UCL),
		1348 Louvain-la-Neuve, Belgium
		(\email{Yurii.Nesterov@uclouvain.be}),
		ORCID: 0000-0002-0542-8757.
	} }

\usepackage{amsopn}

\ifpdf
\hypersetup{
  pdftitle={Contracting Proximal Methods},
  pdfauthor={Nikita Doikov and Yurii Nesterov}
}
\fi

\floatstyle{plaintop}
\restylefloat{algorithm}

\begin{document}

\maketitle

\begin{abstract}
  In this paper, we propose new accelerated methods for smooth convex
  optimization, called contracting proximal methods.
  At every step of these methods, we need to minimize
  a contracted version of the objective function augmented
  by a regularization term in the form of
  Bregman divergence.
  We provide global convergence analysis for a general scheme
  admitting inexactness in solving the auxiliary subproblem.
  In the case of using for this purpose high-order tensor methods,
  we demonstrate an acceleration effect
  for both convex and uniformly convex composite objective
  functions. Thus, our construction explains acceleration for methods of any order
  starting from one. The augmentation of the number of calls of
  oracle due to computing the contracted proximal steps is
  limited by the logarithmic factor in the worst-case complexity
  bound.
\end{abstract}

\begin{keywords}
  convex optimization, proximal method,
  accelerated methods, global complexity bounds,
  high-order algorithms
\end{keywords}

\begin{AMS}
	49M15, 49M37, 65K05, 90C25, 90C30
\end{AMS}

\begin{DOI}
	10.1137/19M130769X
\end{DOI}

\section{Introduction} \label{SectionIntroduction}

One of the classical iterative methods in theoretical
optimization is the proximal point
algorithm~\cite{rockafellar1976monotone}. This method, as
applied to minimizing a convex function $f: \dom f \to
\R$, consists of solving at each iteration the following
subproblem:
\beq \label{ProxMethod}
\ba{rcl}
x_{k + 1} & = & \argmin\limits_{x}
\Bigl\{
a_{k + 1} f(x) + \frac{1}{2}\|x_k - x\|^2
\Bigr\},
\qquad k \geq 0,
\ea
\eeq
where $\|\cdot\|$ is the standard Euclidean norm, and $\{
a_k \}_{k \geq 1}$ is a sequence of positive coefficients.
In general, we can hope only to use an inexact solution of
the subproblem~\cref{ProxMethod}
(see~\cite{guler1991convergence,solodov2001unified,schmidt2011convergence}
for the convergence analysis). An important observation is
that the regularized objective in~\cref{ProxMethod} is
\textit{strongly convex}. Therefore, we can hope that
computing an (inexact) proximal step is usually simpler
than solving the initial problem.

For a function $f \in \mathscr{F}^{1,1}_L$ (convex
differentiable functions with Lipschitz continuous
gradients), we can set all values of the coefficients
$a_k$ equal to a positive constant. This gives a global
sublinear rate of convergence of the iterations~\cref{ProxMethod} in
functional residual of the order~$O(1/k)$. 
This rate is the same rate as that of
the gradient method~\cite{nesterov2018lectures}.

For the same class of functions, we can get a faster rate
of convergence of the order $O(1/k^2)$ using the
accelerated gradient method~\cite{nesterov1983method}.
This is the best possible rate achievable for the
first-order black-box optimization on
$\mathscr{F}^{1,1}_L$~\cite{nemirovskii1983problem}. An
accelerated variant of the proximal point algorithm with
the optimal rate of convergence was proposed
in~\cite{guler1992new} (see
also~\cite{salzo2012inexact,lin2015universal,lin2018catalyst,ivanova2019adaptive} 
for extensions and some applications).

In this paper, we present a new family of proximal-type
algorithms for smooth convex optimization called
\textit{contracting proximal methods}, which includes an
accelerated algorithm from~\cite{guler1992new} as a
particular case and provides a systematic way for
constructing faster proximal accelerated methods for
high-order optimization. Thus, for the class of convex
functions, whose $p$-th derivative is Lipschitz continuous
($p \geq 1$), our new methods achieve the $O(1 / k^{p +
	1})$-rate of convergence for the outer proximal
iterations, while the inner subproblems can be efficiently
solved up to desired accuracy by the high-order tensor
methods~\cite{nesterov2018implementable}.
Note that this rate can also be achieved by 
a direct acceleration scheme, utilizing the notion of
estimating sequences~\cite{baes2009estimate, nesterov2018implementable}.
It can be improved up to the level $O(1 / k^{ \frac{3p + 1}{2} })$
by using a special line-search on each iteration~\cite{monteiro2013accelerated,gasnikov2019near}.
The latter rate was shown to be the optimal one~\cite{arjevani2019oracle}.

The main difference between contracting proximal methods
and the classical approach~\cref{ProxMethod} consists in
employing the \textit{contracted} objective function
(which provides the methods with their name) and the
\textit{Bregman divergence} (notation $\beta_d(x; y)$)
instead of the usual Euclidean norm. The exact form of our
method is very simple:
\beq \label{ContrMethod}
\left. \ba{rcl} v_{k + 1} & = & \argmin\limits_{x} \Bigl\{
A_{k + 1} f\bigl( \frac{a_{k + 1}x + A_k x_k}{A_{k + 1}}
\bigr) +
\beta_d(v_k; x) \Bigr\} \\
\\
x_{k + 1} & = & \frac{a_{k + 1} v_{k + 1} + A_k x_k}{A_{k
		+ 1}} \ea \right\},  \quad k \geq 0.
\eeq
Thus, we use a sequence of auxiliary points $\{ v_k \}_{k
	\geq 0}$ and the scaling coefficients $A_k \Def \sum_{i =
	1}^k a_i$.

Let us illustrate the basic idea behind this construction
by the simplest \textit{Euclidean setting}, when
$\beta_d(x;y) \equiv \frac{1}{2}\|x - y\|^2$. We are going
to ensure at each iteration $k \geq 0$ the following
condition:
\beq \label{IntroMainIneq}
\ba{rcl}
\frac{1}{2}\|x_0 - x\|^2 + A_k f(x) & \geq & \frac{1}{2}\|v_k - x\|^2 + A_k f(x_k),
\qquad x \in \dom f.
\ea
\eeq
A direct consequence of~\cref{IntroMainIneq} is the
global convergence bound
\beq \label{IntroConverg}
\ba{rcl}
f(x_k) - f^{*} & \leq & \frac{\|x_0 - x^{*}\|^2}{2 A_k}.
\ea
\eeq
We can propagate inequality~\cref{IntroMainIneq} to the
next iteration by a trivial observation:
$$
\ba{rcl}
\frac{1}{2}\|x_0 - x\|^2 + A_{k + 1} f(x) & = &
\frac{1}{2}\|x_0 - x\|^2 + A_k f(x) + a_{k + 1} f(x) \\
\\
& \refGE{IntroMainIneq} &
\frac{1}{2}\|v_k - x\|^2 + A_k f(x_k) + a_{k + 1} f(x) \\
\\
& \geq & \frac{1}{2}\|v_k - x\|^2
+ A_{k + 1} f\bigl( \frac{a_{k + 1} x + A_k x_k}{A_{k + 1}}  \bigr)
\; \equiv \; h_{k + 1}(x),
\ea
$$
where the last inequality is due to convexity of the
objective. Note that the first step of contracting
proximal method~\cref{ContrMethod} is defined exactly as
follows:
\beq\label{prob-V}
\ba{rcl}
v_{k + 1} & = & \argmin\limits_{x \in \E} h_{k + 1}(x).
\ea
\eeq
Hence, by strong convexity of $h_{k + 1}(\cdot)$, we
finally justify that
$$
\ba{rcl}
h_{k + 1}(x) & \geq & h_{k + 1}(v_{k + 1}) + \frac{1}{2}\|v_{k + 1} - x\|^2
\; \geq \; A_{k + 1}f(x_{k + 1}) + \frac{1}{2}\|v_{k + 1} - x\|^2.
\ea
$$

Thus, for the Euclidean setting,
iteration~\cref{ContrMethod} immediately results in the
convergence guarantee~\cref{IntroConverg}. However, we
are still free in the choice of coefficients $\{ a_k \}_{k
	\geq 1}$. The only reason for bounding their growth
consists in keeping the complexity of the optimization
problem~\cref{prob-V} at an acceptable
level.\footnote{Hence, these bounds should take into
	account the efficiency of the auxiliary minimization
	scheme used for solving the problem (\ref{prob-V}).} For
$f \in \mathscr{F}^{1,1}_L$, the recommended choice of
$a_{k + 1}$ corresponds to the quadratic equation
\cite{nesterov1983method}:
\beq\label{def-AK}
\ba{rcl}
a_{k + 1}^2  & = & \frac{1}{L}(a_{k + 1} + A_k).
\ea
\eeq
It is easy to see that this choice results in the optimal
$O(1 / k^2)$-rate of convergence for the method. On the
other hand, it makes the \textit{condition number} of the
problem \cref{prob-V} equal to an absolute constant.
Let us assume for simplicity that $f$ is two times continuously differentiable.
Then, in view of the presence of the regularization
term, $\nabla^2 h_{k+1}(x) \succeq I$.
On the other hand,
$$
\ba{rcl}
\nabla^2 h_{k+1}(x) & = & I + {a_{k+1}^2 \over A_{k+1}}
\nabla^2 f\bigl( \frac{a_{k + 1} x + A_k x_k}{A_{k + 1}}
\bigr) \; \overset{\cref{def-AK}}{\preceq} \; 2 I.
\ea
$$
Hence, we are able to solve the problem \cref{prob-V}
very efficiently by a usual gradient method (see the
details in~\cref{SectionTensorMethods}).

It is remarkable that exactly the same reasoning justifies
the accelerated versions of {\em all} high-order tensor
methods ($p \geq 2$). The only difference consists in the
degree of the proximal term, which must be compatible with
the order of optimization scheme used for solving the
problem~\cref{prob-V}.

Our first-order contracting proximal method for the Euclidean
setting (described above) produces the same sequence of
points as the accelerated Proximal Point Algorithm
from~\cite{guler1992new}. However, now we can also employ 
the Bregman divergence, which sometimes is more suitable
to the topology of our function and ensures faster
convergence.

The rest of the paper is organized as follows.
\Cref{SectionNotation} introduces notation used
throughout the paper and describes our problem of interest
in the composite form. We also give a definition of
Bregman divergence and mention some of its properties.

In~\cref{SectionProximalMethod}, we introduce a
general contracting proximal method (formulated 
as~\cref{MainAlgorithm}). 
We present its convergence analysis for a
problem in composite form and arbitrary Bregman
divergence. We study both convex and strongly convex cases
under inexactness in proximal steps.
\cref{TheoremGlobal} specifies how the parameters
of the algorithm and inner accuracy affect the convergence
rate.

In~\cref{SectionTensorMethods}, we discuss
implementation of one iteration of our method, under
the assumption that 
the $p$-th derivative ($p \geq 1$) of the
smooth part of the objective is Lipschitz continuous. 
We present a fully defined optimization scheme (\cref{TensorAlgorithm}),
with incorporated steps of the tensor method of a certain
degree. The resulting algorithm achieves the accelerated rate
of convergence, with an additional logarithmic factor for
the number of total oracle calls. The final complexity
estimate for this scheme is given by~\cref{TheoremFinal1,TheoremFinal2}. 

\Cref{SectionNumerical} contains numerical experiments.
\Cref{SectionConclusion} has some final remarks.

\section{Notation}
\label{SectionNotation}\SetEQ

In what follows, we denote by $\E$ a finite-dimensional
real vector space and by $\E^{*}$ its dual space, which is
a space of linear functions on $\E$. The value of function
$s \in \E^{*}$ at point $x \in \E$ is denoted by $\langle
s, x \rangle$.

Let us fix some arbitrary (possibly non-Euclidean) norm $\| \cdot \|$
on space $\E$ and define the dual norm $\| \cdot \|_{*}$ on $\E^{*}$ in
the standard way:
$$
\ba{rcl}
\|s\|_{*} &  \Def &  \sup\limits_{h \in \E} \{ \la s, h \ra : \|h\| \leq 1 \}.
\ea
$$

For a smooth function $f$, its gradient at point $x$ is
denoted by $\nabla f(x)$, and its Hessian
is $\nabla^2f(x)$. Note that
$$
\ba{rcl}
\nabla f(x) & \in & \E^{*},
\qquad
\nabla^2 f(x) h \;\, \in \;\, \E^{*},
\qquad
x \in \dom f, \;\; h \in \E.
\ea
$$
Higher derivatives
are denoted as $D^p f(x)[\cdot]$, which are $p$-linear
symmetric forms on $\E$,
and the norm is induced:
$$
\ba{rcl}
\| D^p f(x) \| & \Def &
\sup\limits_{h_1, \dots, h_p \in \E}
\bigl\{ D^p f(x)[h_1, \dots, h_p]: \; \|h_i\| \leq 1, i = 1, \dots, p
\bigr\}.
\ea
$$

For convex but
not necessary differentiable function $\psi$, we denote by
$\partial \psi(x) \subset \E^{*}$ its subdifferential at
point $x \in \dom \psi$.

Our goal is to solve the following composite minimization problem:
\begin{equation} \label{MainProblem}
\min_{x \in \dom F} \Bigl\{ F(x) \equiv f(x) + \psi(x) \Bigr\},
\end{equation}
where $f$ is several times differentiable on its open
domain convex function, with some reasonable assumptions
on the growth of its derivatives (for example, that its
$p$-th derivative is Lipschitz continuous for some $p \geq
1$), and $\psi: \E \to \R \cup \{ +\infty \}$ is a proper
closed convex function, which we assume to be
\textit{simple}, but possibly nondifferentiable, with
$\dom \psi \subset \dom f$. We also assume that solution
$x^{*} \in \dom F$ of problem~\cref{MainProblem} does
exist, denoting $F^{*} = F(x^{*})$.

Let us fix arbitrary differentiable strictly convex function
$d: \dom \psi \to \R$, which we call 
the\textit{prox function}.
Then, we denote by $\beta_{d}(x; y)$ the corresponding
\textit{Bregman divergence}, centered at $x$:
$$
\ba{rcl}
\beta_{d}(x; y) & \Def &
d(y) - d(x) - \la \nabla d(x), y - x \ra.
\ea
$$

We say that function $d$ is \textit{uniformly convex} of degree $p + 1$ 
(with respect to the norm $\| \cdot \|$) with constant $\sigma_{p + 1}(d) > 0$
if it holds for all $x, y \in \dom d$:
\beq \label{UniformConvexity}
\ba{rcl}
\beta_d(x; y) & \geq & \frac{\sigma_{p + 1}(d)}{p + 1}\|x - y\|^{p + 1}.
\ea
\eeq

The main example, which naturally appears in tensor methods
(see~\cite{nesterov2018implementable}) and which we use
in~\cref{SectionTensorMethods}, is the following
prox function.
\BE
$$
\ba{rcl}
d(x) &  \equiv & \frac{1}{p + 1}\|x - x_0\|^{p + 1}
\ea
$$
for some $p \geq 1$. For the Euclidean norm
(when $\|x\| \equiv \la Bx, x \ra^{1/2}$ for a fixed postive-definite linear operator
$B = B^{*} \succ 0$) this prox function is
\textit{uniformly convex} of degree $p + 1$ with constant~$2^{1 - p}$
(see Lemma~5 in~\cite{doikov2019minimizing}), so it holds
that 
\beq \label{ProxUnifConv}
\ba{rcl}
\beta_d(x; y) & \geq & \frac{2^{1 - p}}{p + 1}\|x - y\|^{p + 1}, \qquad x, y \in \E.
\ea
\eeq
\EE
For more examples of available prox functions see~\cite{bauschke2016descent,lu2018relatively}.

The definition of Bregman divergence can be extended onto
nondifferentiable function $\psi$ by specifying a
particular subgradient $\psi'(x) \in \partial \psi(x)$:
$$
\ba{rcl}
\beta_{\psi}(x, \psi'(x); y) & \Def &
\psi(y) - \psi(x) - \la \psi'(x), y - x \ra.
\ea
$$
However, we will use simpler notation $\beta_{\psi}(x; y)$
if no ambiguity arises.

We say that function $\psi$ is \textit{strongly convex
with respect to} $d$ (see \cite{van2017forward,bauschke2016descent,lu2018relatively}) with
constant $\sigma_d(\psi) > 0$
if it holds for all $x, y \in \dom \psi$ and for all
$\psi'(x) \in \partial \psi(x)$ that
\beq \label{StronglyConvexDefinition}
\ba{rcl}
\beta_{\psi}(x, \psi'(x); y) & \geq & \sigma_d(\psi) \beta_d(x; y).
\ea
\eeq
Inequality~\eqref{StronglyConvexDefinition} always holds
with $\sigma_d(\psi) = 0$  just by convexity. An
interesting illustration of this concept is given by a
regularized Taylor polynomial of degree~$3$ for convex
function (see \cite{nesterov2018implementable}).
\BE
Let $f: \dom f \to \R$ be convex, with Lipschitz continuous third derivative:
$$
\ba{rcl}
\| D^3 f(y) - D^3 f(x) \| \; \leq \; L_3\|y - x\|, \qquad x, y \in \dom f.
\ea
$$
Denote by $\Omega_3(f, x; y)$ its Taylor approximation of degree~$3$
around some fixed point $x$,
$$
\ba{c}
\Omega_3(f, x; y)  \Def  f(x) + \la \nabla f(x), y - x \ra
+ \frac{1}{2} D^2 f(x)[y - x]^2
+ \frac{1}{6}D^{3}f(x)[y - x]^3,
\ea
$$
and consider its regularization of degree $4$, with some $\tau > 1$:
$$
\ba{rcl}
g(y) & \equiv & \Omega_3(f, x; y) + \frac{\tau^2 L_3}{8}\|y - x\|^4.
\ea
$$
Then, for the Euclidean norm, the function $g(\cdot)$
is strongly convex with respect
to the following prox function
(see Lemma~4 in~\cite{nesterov2018implementable}):
$$
\ba{rcl}
d(h) & \equiv &
\frac{1}{2} \left(1 - \frac{1}{\tau} \right)
D^2 f(x)[h]^2
+ \frac{\tau(\tau - 1) L_3}{8}\|h\|^4.
\ea
$$
\EE

Let us summarize some basic properties of Bregman
divergence, which follow directly from its definition. For
any pair $f_1, f_2$ of convex functions and all $x, y \in
\dom (f_1 + f_2)$ we have
\beq \label{BregmanSum}
\ba{rcl}
\beta_{a_1f_1 + a_2f_2}(x; y) & = & a_1\beta_{f_1}(x; y) +
a_2\beta_{f_2}(x; y), \qquad a_1, a_2 \geq 0.
\ea
\eeq
For any linear function $\ell(x) = a + \la g, x \ra$ we
have
\beq \label{BregmanLinear}
\ba{rcl}
\beta_{\ell}(x; y) & = & 0.
\ea
\eeq
Therefore, from~\cref{BregmanSum} and~\cref{BregmanLinear} we conclude that
\beq \label{BregmanBregman}
\ba{rcl}
\beta_f(x; y) & = & \beta_d(x; y),
\ea
\eeq
when $f(y) = \beta_d(z; y)$ for some fixed $z$. Now, consider the following
simple but general construction, which we use in a core of our analysis.
Let $h$ be a regularized composite objective:
$$
\ba{rcl}
h(y) & = & g(y) + a\psi(y) + \gamma \beta_d(z; y), \qquad a, \gamma \geq 0,
\ea
$$
where $g$ and $\psi$ are arbitrary closed convex functions, and $\psi$
is strongly convex with respect to $d$ with some constant $\sigma_d(\psi) \geq 0$.
Then we have, for every $x, y \in \dom h$ and every $h'(x) \in \partial h(x)$
\beq  \label{MainLemma}
\ba{rcl}
h(y) - h(x) - \la h'(x), y - x \ra & = & \beta_h(x; y) \\
\\
& \stackrel{\cref{BregmanSum},\cref{BregmanBregman}}{=} &
\beta_g(x; y) + a \beta_{\psi}(x; y) + \gamma \beta_{d}(x; y) \\
\\
& \geq & (a \sigma_d(\psi) + \gamma) \beta_d(x; y).
\ea
\eeq
In particular, for the exact minimum $T =
\argmin\limits_{y \in \E} h(y)$, we have
\beq  \label{MainLemmaExactMinimum}
\ba{rcl}
h(y) & \geq & h(T) + (a \sigma_d(\psi) + \gamma) \beta_d(T; y).
\ea
\eeq

\section{Contracting proximal method}
\label{SectionProximalMethod}\SetEQ

In our general scheme, we are going to maintain the
following inequality, for every $x \in \dom \psi$ and 
$k \geq 0$:
\beq \label{AcceleratedGuarantee}
\ba{rcl}
\gamma_0 \beta_d(x_0; x) + A_k F(x) & \geq & \gamma_k
\beta_d(v_k; x) + A_k F(x_k) + C_k(x),
\ea
\eeq
where $\{ x_k \}_{k \geq 0}$ and $\{ v_k \}_{k \geq 0}$ are sequences
of points from $\dom \psi$,
$\{ A_k \}_{k \geq 0}$
is a sequence of increasing numbers,
$$
\ba{rcccl}
a_{k + 1} & \Def & A_{k + 1} - A_k \; > \; 0,
\qquad
A_0 & = & 0,
\ea
$$
and $\{ \gamma_k \}_{k \geq 0}$
is a sequences of nondecreasing proximal coefficients,
$$
\ba{rcccl}
\gamma_{k + 1} & \geq & \gamma_k,
\qquad
\gamma_0 & > & 0.
\ea
$$
We would prefer to have functions $C_k(x)$ as big as possible.
Thus, if it happens to be $C_k(x^{*}) \geq 0$ for all $k \geq 1$,
then from~\cref{AcceleratedGuarantee}
we have a convergence guarantee,
$$
\ba{rcl}
F(x_k) - F^{*} & \leq & \frac{\gamma_0 \beta_d(x_0, x^{*})}{A_k},
\qquad k \geq 1,
\ea
$$
and the rate of convergence is determined by the growth of
coefficients $A_k$ toward infinity. However, generally
$C_k(x)$ may have arbitrary sign.

\bigskip

Let us discus a simple possibility for propagating
relation~\cref{AcceleratedGuarantee} to the next
iteration.
\beq \label{TriangleAnalysis}
\ba{rcl}
& & \mspace{-72mu}
\gamma_0 \beta_d(x_0; x) + A_{k + 1} F(x) \\
\\
& = &
\gamma_0 \beta_d(x_0; x) + A_k F(x) + a_{k + 1} F(x) \\
\\
& \overset{\cref{AcceleratedGuarantee}}{\geq} &
\gamma_k \beta_d(v_k; x) + A_k F(x_k) + a_{k + 1} F(x) + C_k(x)  \\
\\
&  \geq &
\gamma_k \beta_d(v_k; x)
+ A_{k + 1}f\bigl( \frac{a_{k + 1}x + A_k x_k}{A_{k + 1}} \bigr)
+ a_{k + 1} \psi(x)
+ A_k \psi(x_k)
+ C_k(x),
\ea
\eeq
where the last inequality is due to the convexity of $f$. Let
us consider a contracted objective with a regularizer from
the last step:
\beq \label{RegularizedObjective}
\ba{rcl}
h_{k + 1}(x)
& \Def &
A_{k + 1}f\bigl( \frac{a_{k + 1}x + A_k x_k}{A_{k + 1}} \bigr)
+ a_{k + 1} \psi(x)
+ \gamma_k \beta_d(v_k; x).
\ea
\eeq
This function is strongly convex with respect to
$d(\cdot)$ with parameter
\beq \label{GammaKDef}
\ba{rcl}
\sigma_{d}(h_{k + 1})
& \geq &
\gamma_{k + 1} \; \Def \;
a_{k + 1} \sigma_{d}(\psi) + \gamma_k.
\ea
\eeq
If we are able to compute the \textit{exact minimum}
\beq \label{InnerSubproblem}
\ba{rcl}
T & = & \argmin\limits_{x \in \E} h_{k + 1}(x),
\ea
\eeq
then by~\cref{MainLemmaExactMinimum} we see that
$$
\ba{rcl}
& & \mspace{-72mu}
h_{k + 1}(x) + A_k \psi(x_k) \\
\\
& \geq &
h_{k + 1}(T) +
\gamma_{k + 1} \beta_{d}(T; x) + A_k \psi(x_k) \\
\\
& = &
A_{k + 1}f\bigl( \frac{a_{k + 1}T + A_k x_k}{A_{k + 1}} \bigr)
+ a_{k + 1} \psi(T)
+ \gamma_k \beta_d(v_k; T)
+ \gamma_{k + 1} \beta_d(T; x) + A_k \psi(x_k) \\
\\
& \geq &
A_{k+1} F\bigl( \frac{a_{k + 1}T + A_k x_k}{A_{k + 1}} \bigr)
+ \gamma_k \beta_d(v_k; T) + \gamma_{k + 1} \beta_d(T; x).
\ea
$$
And it is natural to set $v_{k + 1} = T$ and
\beq \label{ConvexComb}
\ba{rcl}
x_{k + 1} & \Def & \frac{a_{k + 1}v_{k + 1} + A_k x_k}{A_{k + 1}}.
\ea
\eeq
Thus we would obtain guarantee~\cref{AcceleratedGuarantee} for the next step,
with
$$
\ba{rcl}
C_{k + 1}(x) & \equiv & C_k(x) + \gamma_k \beta_{d}(v_k; v_{k + 1})
\; \equiv \; \sum_{i = 1}^k \gamma_i \beta_d(v_i; v_{i + 1}) \geq 0.
\ea
$$

\bigskip

Now, instead of computing the exact
minimum~\cref{InnerSubproblem}, let us relax $v_{k + 1}
\in \dom \psi$ to be a point with a \textit{small norm of
	subgradient}:
\begin{equation} \label{InexactCondition}
\|s\|_{*} \; \leq \; \delta_{k + 1}
\quad
\text{for some}
\quad s \in \partial h_{k + 1}(v_{k + 1}).
\end{equation}
Note that condition~\cref{InexactCondition} can be easily
verified algorithmically since in composite setting we are
able to compute points with a small subgradient of $h_{k +
	1}$ (see \cite{nesterov2018implementable}).

Thus, we come to the following general scheme.

\vspace*{7pt}
\fbox{
	\hspace*{4pt}
	\centering
	\begin{minipage}{0.8\linewidth} 
		\begin{algorithm}[H] 
			\renewcommand{\thealgorithm}{1}
			\caption{\textbf{Contracting Proximal Method}}
			\label{MainAlgorithm}
			\noindent\makebox[\linewidth]{\rule{1.093\linewidth}{0.4pt}}
			\begin{algorithmic}[1]
				\vspace{5pt}
				\Require Choose $x_0 \in \dom \psi$, $\gamma_0 > 0$, set $v_0 := x_0$, $A_0 := 0$.
				\vspace*{5pt}
				\Ensure $k \geq 0$. \vspace*{5pt}
				\State Choose $a_{k + 1} > 0$. Set $A_{k + 1} := A_k + a_{k + 1}$. \vspace*{5pt}
				\State Denote contracted objective with regularizer: \vspace*{5pt}
				\Statex
				$\qquad h_{k + 1}(x) := A_{k + 1}f\bigl(\frac{a_{k + 1}x + A_k x_k}{A_{k + 1}} \bigr)
				+ a_{k + 1} \psi(x) + \gamma_k \beta_d(v_k; x)$.
				\vspace*{5pt}
				\State Choose accuracy $\delta_{k + 1} \geq 0$. \vspace*{5pt}
				\State Find $v_{k + 1} \in \dom \psi$ such that
				$\exists \, s \in \partial h_{k+1}(v_{k + 1}): \; \|s\|_{*} \leq \delta_{k + 1}$. \vspace*{5pt}
				\State Set $x_{k + 1} := \frac{a_{k + 1}v_{k + 1} + A_k x_k}{A_{k + 1}}$.  \vspace*{5pt}
				\State Set $\gamma_{k + 1} := \gamma_k + a_{k + 1} \sigma_d(\psi)$.
				\vspace*{5pt}
			\end{algorithmic}
		\end{algorithm}
	\end{minipage}
	\hspace*{4pt}
}
\vspace*{10pt}

At this moment, we need one additional assumption. It
relates the dual norm $\| \cdot \|_{*}$ (used at step~4)
with the Bregman divergence $\beta_d(v; x)$.

\BAS \label{AssumptionUniformlyConvex}
For some $p \geq 1$, prox-function $d(\cdot)$ is uniformly
convex of degree $p + 1$ with respect to the primal norm $\| \cdot \|$
with parameter $\sigma_{p + 1}(d)
> 0$ (see inequality~(\ref{UniformConvexity})).
\EAS

Let us write down the convergence guarantees of the
method.

\BT[convergence of contracting proximal method] \label{TheoremGlobal}
Let~\cref{AssumptionUniformlyConvex} hold. Then
for~\cref{MainAlgorithm} at all iterations $k \geq 0$ we have
\beq \label{InexactGuarantee}
\ba{rcl}
A_k \left( F(x_k) - F^{*} \right)
+ \gamma_k \beta_d(v_k; x^{*})
+  \sum\limits_{i = 1}^k \gamma_i\beta_d(v_{i - 1}; v_i)
& \leq & R_k(p, \delta),
\ea
\eeq
where
\beq \label{RkDef}
\ba{rcl}
R_k(p, \delta)
& \Def &
\left( \bigl( \gamma_0 \beta_d(x_0; x^{*}) \bigr)^{p \over p + 1} +
\left( { p + 1 \over \sigma_{p + 1}(d) } \right)^{1 \over p + 1}
\sum\limits_{i = 1}^k
{\delta_i \over \gamma_i^{1 / (p + 1)}} \right)^{p + 1 \over p}.
\ea
\eeq
\ET
\begin{proof}
First, let us ensure by induction in $k \geq 0$ the
following inequality:
\beq \label{AcceleratedGuarantee2}
\ba{rc}
& A_k \left( F(x_k) - F(x) \right)
+ \gamma_k \beta_d(v_k; x)
+ \sum\limits_{i = 1}^k \gamma_i \beta_d(v_{i - 1}; v_i) \\
\\
& \quad \leq \quad \gamma_0\beta_d(x_0; x) +
\sum\limits_{i = 1}^k \la s_i, v_i - x \ra, \quad x \in
\dom \psi,
\ea
\eeq
where $s_i \in \partial h_i(v_i)$. It is obviously true
for $k = 0$. Let it hold for some $k \geq 0$ and consider
the next. Note that~\cref{AcceleratedGuarantee2} is
exactly \cref{AcceleratedGuarantee} with
$$
\ba{rcl}
C_k(x) & \equiv &
\sum\limits_{i = 1}^k \Bigl[
\gamma_i \beta_d(v_{i - 1}; v_i) + \la s_i, x - v_i \ra
\Bigr].
\ea
$$
Therefore, we have
$$
\ba{rcl}
& & \mspace{-72mu}
\gamma_0 \beta_d(x_0; x) + A_{k+1} F(x) \\
\\
& \overset{\cref{TriangleAnalysis}}{\geq} &
h_{k + 1}(x) + A_k \psi(x_k) +  C_k(x) \\
\\
& \overset{\cref{MainLemma}}{\geq} &
h_{k + 1}(v_{k + 1}) + \la s_{k + 1}, x - v_{k + 1} \ra
+ \gamma_{k + 1} \beta_d(v_{k + 1}; x) + A_k \psi(x_k) +  C_k(x) \\
\\
& = &
A_{k + 1} f(x_{k + 1})
+ a_{k + 1} \psi(v_{k + 1}) + \gamma_{k + 1} \beta_d(v_{k + 1}; x)
+ A_k \psi(x_k) + C_{k + 1}(x) \\
\\
& \geq & A_{k + 1} F(x_{k + 1}) + \gamma_{k + 1}
\beta_d(v_{k + 1}; x) + C_{k + 1}(x).
\ea
$$
This is \cref{AcceleratedGuarantee2} for the next step.

Now, plugging $x \equiv x^{*}$
into~\cref{AcceleratedGuarantee2} and taking into account
the nonnegativity of all terms in the left-hand side, we get
$$
\ba{rcl}
\gamma_k \beta_d(v_k; x^{*}) & \leq & \gamma_0
\beta_d(x_0; x^{*}) + \sum\limits_{i = 1}^k \la s_i, v_i -
x^{*} \ra.
\ea
$$
Now, we need to estimate the right-hand side from above.
Using the uniform convexity~\cref{UniformConvexity}, we
conclude that for every $k \geq 0$
\beq \label{RecurrBound}
\ba{rcl}
{\gamma_k \sigma_{p + 1}(d) \over p + 1}\|v_k - x^{*}\|^{p + 1}
& \leq &
\gamma_0 \beta_d(x_0; x^*)
+ \sum\limits_{i = 1}^k \|s_i\|_* \cdot \|v_i - x^{*}\| \\
\\
& \overset{\cref{InexactCondition}}{\leq} &
\gamma_0 \beta_d(x_0; x^*)
+ \sum\limits_{i = 1}^k \delta_i \|v_i - x^{*}\|
\quad \equiv \quad \alpha_k.
\ea
\eeq
In order to finish the proof, it is enough to bound from
above the value $\alpha_k$, for which we have the
following recurrence:
$$
\ba{rcl}
\alpha_k & = & \alpha_{k - 1} + \delta_k \|v_k - x^{*}\|
\quad \overset{\cref{RecurrBound}}{\leq} \quad
\alpha_{k - 1} +
{\delta_k  }
\left(
{ p + 1 \over \gamma_k \sigma_{p + 1}(d)
} \right)^{1 \over p + 1}
\alpha_{k}^{1 \over p + 1}.
\ea
$$
Dividing both sides by $\alpha_k^{1 \over p + 1}$ and
using the monotonicity of this sequence, we get
$$
\ba{rcl}
\alpha_{k}^{p \over p + 1}
& \leq &
\frac{\alpha_{k - 1}}{\alpha_k^{1 / (p + 1) }}
+ {\delta_k}
\left(
{ p + 1 \over \gamma_k \sigma_{p + 1}(d)  }
\right)^{1 \over p + 1}
\quad \leq \quad
\alpha_{k - 1}^{p \over p + 1}
+ \delta_k
\left(
{ p + 1 \over \gamma_k \sigma_{p + 1}(d)  }
\right)^{1 \over p + 1}.
\ea
$$
Finally, from the last inequality we obtain
$$
\ba{rcl}
\alpha_k
& \leq &
\left(
\alpha_0^{p \over p + 1}
+ \left( {p + 1 \over \sigma_{p + 1}(d) } \right)^{1 \over p + 1}
\sum\limits_{i = 1}^k {\delta_i \over \gamma_i^{1 / (p + 1)}}
\right)^{p + 1 \over p},
\ea
$$
which is the right-hand side of~\cref{InexactGuarantee}.
\end{proof}

We see that accuracies $\delta_k$ for subgradients of the
subproblems appear in~\cref{RkDef} in an additive form,
weighted by the coefficients $\gamma_k^{-{1 \over p +
		1}}$. They should be chosen in a way making the right-hand
side of~\cref{InexactGuarantee} small enough. Let us
consider the simplest case, when all $\delta_k$ are the
same.

\begin{corollary} \label{ChooseAccuracies}
	Let $\delta_k = \delta>0$ for all $k \geq 1$. Assume that
	the coefficients $A_k$ grow \normalfont{sublinearly}:
	\beq \label{AkGrowth}
	\ba{rcl}
	A_k & \geq & c k^{p + 1}, \quad k \geq 1,
	\ea
	\eeq
	with some constant $c > 0$. Then for every
	\beq \label{DeltaChoice1}
	\ba{rcl}
	k & \geq & \left(
	{ \gamma_0 \beta_{d}(x_0; x_{*})  \over c \varepsilon}
	\right)^{\frac{1}{p + 1}} 2^{\frac{1}{p}}
	\qquad \text{and} \qquad
	\delta \; \; \leq \; \;
	\frac{(c\varepsilon)^{\frac{p}{p + 1}} }{ 2 }
	\left({ \gamma_0 \sigma_{p + 1}(d) \over p + 1 }\right)^{1 \over p + 1}
	\ea
	\eeq
	we have
	\beq \label{DeltaChoiceRes1}
	\ba{rcl}
	R_k(p, \delta) & \leq & \varepsilon A_k.
	\ea
	\eeq
	Consequently, by~\cref{InexactGuarantee} we have $ F(x_k)
	- F^{*} \leq \varepsilon. $
\end{corollary}
\begin{proof}
Indeed,
$$
\ba{rcl}
\left(
{ \gamma_0 \beta_d(x_0; x^{*}) \over A_k }
\right)^{\frac{p}{p + 1}}
& \overset{\cref{AkGrowth}}{\leq} &
\left(
\frac{\gamma_0 \beta_d(x_0; x^{*})}{c}
\right)^{p \over p + 1} k^{p}
\; \overset{\cref{DeltaChoice1}}{\leq}
\; \frac{\varepsilon^{\frac{p}{p + 1}}}{2}
\ea
$$
and
$$
\ba{rcl}
\frac{
	\left(
	{ p + 1 \over \sigma_{p + 1}(d) }
	\right)^{\frac{1}{p + 1}} }
{A_k^{p \over p + 1}}
\sum\limits_{i = 1}^k
{\delta_i \over \gamma_i^{1 / (p + 1)}}
& \leq &
\frac{
	\left(
	{ p + 1 \over \gamma_0 \sigma_{p + 1}(d) }
	\right)^{\frac{1}{p + 1}} k \delta }
{A_k^{p \over p + 1}}
\; \overset{\cref{AkGrowth}}{\leq} \;
\frac{
	\left(
	{ p + 1 \over \gamma_0 \sigma_{p + 1}(d) }
	\right)^{\frac{1}{p + 1}}  \delta}
{c^{p \over p + 1} k^{p + 1} } \\
\\
& \leq &
\frac{
	\left(
	{ p + 1 \over \gamma_0 \sigma_{p + 1}(d) }
	\right)^{\frac{1}{p + 1}} \delta  }
{c^{p \over p + 1} }
\; \overset{\cref{DeltaChoice1}}{\leq} \;
\frac{\varepsilon^{\frac{p}{p + 1}}}{2}.
\ea
$$
Summing up these two inequalities
we obtain~\eqref{DeltaChoiceRes1}.
\end{proof}

\begin{corollary} \label{ChooseAccuraciesLinear}
	Let $\delta_k = \delta>0$ for all $k \geq 1$. Let the
	coefficients $A_k$ grow \normalfont{linearly}:
	\beq \label{LinearRateAk}
	\ba{rcl}
	A_k & \geq & A_1 \exp\bigl( \omega (k - 1) \bigr), \quad k
	\geq 1,
	\ea
	\eeq
	with some constant $0 < \omega \leq 1$ and initial $A_1 >
	0$. Then for every
	\beq \label{KChoice2}
	\ba{rcl}
	k & \geq & 1 + \frac{1}{\omega} \log\left(
	\frac{\gamma_0 \beta_d(x_0; x^{*} ) }{A_1 \varepsilon }
	2^{ (p + 1) / p}
	\right)
	\ea
	\eeq
	and
	\beq \label{DeltaChoice2}
	\ba{rcl}
	\delta & \leq &
	\frac{(A_1 \varepsilon)^{\frac{p}{p + 1}} \omega }{
		2  }
	\cdot
	\frac{p}{p + 1}
	\cdot
	\left(
	{ \gamma_0 \sigma_{p + 1}(d) \over p + 1 }
	\right)^{1 \over p + 1}
	\ea
	\eeq
	we have
	\beq \label{DeltaChoiceRes2}
	\ba{rcl}
	R_k(p, \delta) & \leq & \varepsilon A_k.
	\ea
	\eeq
	Consequently, by~\cref{InexactGuarantee} we have $ F(x_k)
	- F^{*} \leq  \varepsilon. $
\end{corollary}
\begin{proof}
Indeed,
$$
\ba{rcl}
\left({
	\gamma_0 \beta_d(x_0; x^{*}) \over A_k
}\right)^{p \over p + 1}
& \overset{\cref{LinearRateAk}}{\leq} &
\left({
	\gamma_0 \beta_d(x_0; x^*) \over
	A_1 \exp\bigl( \omega (k - 1) \bigr)
}\right)^{p \over p + 1}
\; \overset{\cref{DeltaChoice2}}{\leq} \;
\frac{\varepsilon^{p \over p + 1}}{2}.
\ea
$$
Now, note that the following inequality holds for all $x
\geq 0$:
\beq \label{ExpIneq}
\ba{rcl}
\exp(x) & \geq & 1 + x.
\ea
\eeq
Therefore,
\beq \label{AkExpIneq}
\ba{rcl}
\frac{A_k^{p \over p + 1}}{k}
& \overset{\cref{LinearRateAk}}{\geq} &
\frac{A_1^{p \over p + 1}
	\exp\left( \frac{p}{p + 1} \omega (k - 1) \right)}{k}
\; \overset{\cref{ExpIneq}}{\geq} \;
\frac{ A_1^{p \over p + 1}
	\Bigl( 1 + \frac{p}{p + 1}\omega(k - 1) \Bigr) }{ k }
\; > \;
\frac{p}{p + 1} A_1^{p \over p + 1} \omega.
\ea
\eeq
And we obtain
$$
\ba{rcl}
\frac{
	\left(
	{ p + 1 \over \sigma_{p + 1}(d) }
	\right)^{\frac{1}{p + 1}} }
{A_k^{p \over p + 1}}
\sum\limits_{i = 1}^k {\delta_i \over \gamma_i^{1 / (p + 1)}}
& \leq &
\frac{
	\left(
	{ p + 1 \over \gamma_0 \sigma_{p + 1}(d) }
	\right)^{\frac{1}{p + 1}} k \delta }
{A_k^{p \over p + 1}}
\;\; \overset{\cref{AkExpIneq}}{<} \;\;
\frac{
	\left(
	{ p + 1 \over \gamma_0 \sigma_{p + 1}(d)
	}\right)^{\frac{1}{p + 1}} p + 1 \delta}{
	A_1^{p \over p + 1} p \omega} \\
\\
& \overset{\cref{DeltaChoice2}}{\leq} &
\frac{\varepsilon^{p \over p + 1}}{2}.
\ea
$$
\,
\end{proof}

Estimates~\cref{DeltaChoice1} and~\eqref{DeltaChoice2}
show that the bound for the inner accuracy $\delta$ has a
reasonable dependency on the absolute accuracy
$\varepsilon$ required for the initial
problem~\cref{MainProblem}. Thus, in both cases, in
step~4 of the algorithm we need to find a point $v_{k+1}$
with subgradient $s \in \partial h_{k + 1}(v_{k + 1})$:
$$
\ba{rcl}
\|s\|_{*} \; \leq \; O\Bigl(\varepsilon^{p \over p +
	1}\Bigr) & \Leftrightarrow & \|s\|_{*}^{p + 1 \over p} \;
\leq \; O\bigl( \varepsilon \bigr).
\ea
$$
This is a reachable goal, especially for methods
minimizing $h_{k + 1}(\cdot)$ with a linear rate of
convergence.

In practice, it may be reasonable not to use very small
inner accuracy on a first stage but to decrease it over the iterations.
Then, the following simple choice of $\{ \delta_k \}_{k \geq 0}$ can work.

\begin{corollary} \label{ChooseAccuraciesDecreasing}
	Let us define $\delta_k \equiv \frac{c}{k^s}$ with fixed
	absolute constants $c > 0$ and $s > 1$. Then,
	$$
	\ba{rcl}
	\sum\limits_{i = 1}^k \delta_i & = &
	c\Bigl(1 + \sum\limits_{i = 2}^k \frac{1}{i^s} \Bigr)
	\; \leq \; c\Bigl(1 + \int\limits_{1}^{+\infty} \frac{dx}{x^s} \Bigr)
	\; = \; \frac{cs}{s - 1}.
	\ea
	$$
	Therefore, we have
	$$
	\ba{rcl}
	R_k(p, \delta) & \leq &
	\left( \bigl( \gamma_0 \beta_d(x_0; x^{*}) \bigr)^{p \over p + 1}
	+ \Bigl( \frac{p + 1}{\gamma_0 \sigma_{p + 1}(d)} \Bigr)^{1 \over p + 1}
	\frac{cs}{s - 1}
	\right)^{\frac{p + 1}{p}}.
	\ea
	$$
\end{corollary}

\section{Application of tensor methods}
\label{SectionTensorMethods}\SetEQ

In this section, let us incorporate the high-order tensor
methods~\cite{nesterov2018implementable} into~\cref{MainAlgorithm}
for solving the corresponding inner
subproblem~\cref{InnerSubproblem}. From now on, we
restrict our attention to Euclidean norms. Let us fix
symmetric positive-definite linear operator $B: \E \to
\E^{*}$ (notation $B = B^{*} \succ 0$) and use the
following norm for the primal space: $\|x\| \equiv \la Bx,
x \ra^{1/2}$, $x \in \E$. The norms for multilinear forms
on $\E$ are induced in the standard way (see~\cref{SectionNotation}).

\BAS \label{AssumptionLipschitz}
For fixed $p \geq 1$, the $p$-th derivative of the smooth
component of the objective function is Lipschitz
continuous:
\beq \label{LipschitzDef}
\ba{rcl}
\| D^p f(x) - D^p f(y) \| & \leq & L_p(f)\|x - y\|, \quad x, y \in \dom f,
\ea
\eeq
with some constant $0 < L_p(f) < +\infty$.
\EAS

For this setup, we use the following simple prox function:
\beq \label{ProxChoice}
\ba{rcl}
d(x) & \equiv & \frac{1}{p + 1}\|x - x_0\|^{p + 1}.
\ea
\eeq
Thus, the choice of prox function~\cref{ProxChoice} is
strictly related to the preferable degree $p \geq 1$ of
smoothness of function $f$.

Let us define the Taylor approximation $\Omega_p(f, x; y)$
of function~$f$ around the point $x \in \dom f$:
$$
\ba{rcl}
\Omega_p(f, x; y) & \Def & f(x) + \sum\limits_{i = 1}^p
\frac{1}{i!} D^i f(x)[y - x]^i.
\ea
$$
By \cref{AssumptionLipschitz}, we are able to
bound its accuracy in the following way: for all $x, y \in
\dom f$ it holds that
\beq \label{LipF}
\ba{rcl}
|f(y) - \Omega_p(f, x; y)|
& \leq & \frac{L_p(f)}{(p + 1)!}\|y - x\|^{p + 1},
\ea
\eeq
\beq \label{LipG}
\ba{rcl}
\| \nabla f(y) - \nabla_{\!y} \, \Omega_p(f, x; y) \|_{*}
& \leq & \frac{L_p(f)}{p!}\|y - x\|^p.
\ea
\eeq

Let us look at our regularized objective~$h_{k +
	1}(\cdot)$, which needs to be minimized at every step $k
\geq 0$:
\beq \label{SubproblemObjective}
\ba{rcl}
h_{k + 1}(x) & = &
\underbrace{A_{k + 1}f\left({ a_{k + 1}x + A_k x_k \over A_{k + 1} }\right)
}_{ \Def g_{k + 1}(x) }
\; + \;
\underbrace{
	a_{k + 1} \psi (x) + \gamma_{k} \beta_d(v_k; x)
}_{ \Def \phi_{k + 1}(x) }.
\ea
\eeq
This is a sum of two convex functions: smooth component
$g_{k + 1}$ and possibly nonsmooth but simple component
$\phi_{k + 1}$, which is strongly convex with respect to
$d$.

Let us drop unnecessary indices and consider the
subproblem in a general form:
\beq \label{SubproblemSimple}
\min_{x \in \dom h} \Bigl\{
h(x) \equiv g(x) + \phi(x)
\Bigr\}
\eeq
with $g$ having bounded Lipschitz constant for some $p
\geq 1$: $0 < L_p(g) < +\infty$. Since we assume the
objective to be strongly convex with respect to $d$
from~\cref{ProxChoice} with parameter $\sigma_{d}(h) >
0$, for every $x, y \in \dom h$ and all $h'(x) \in
\partial h(x)$ we have
\beq \label{UnifConv}
\ba{rcl}
h(y) - h(x) - \la h'(x), y - x \ra
& \geq & \sigma_{d}(h) \beta_d(x; y)
\; \overset{\cref{ProxUnifConv}}{\geq} \;
\frac{ \sigma_d(h) 2^{1 - p} }{p + 1} \| y - x \|^{p + 1}.
\ea
\eeq
Bound~\cref{LipF} motivates us to define the following
point,
\beq \label{TensorStep}
\ba{rcl}
T_{M}(h; x) & \Def &
\argmin\limits_{y \in \E}
\left\{
\Omega_p(g, x; y) + \frac{M}{(p + 1)!}\|y - x\|^{p + 1}
+ \phi(y)
\right\},
\ea
\eeq
and consider the following iteration process,
\beq \label{TensorIterations}
\boxed{
	\ba{rcl}
	z_{t + 1} & = &  T_M(h; z_t), \qquad t \geq 0.
	\ea
}
\eeq

For $p = 1$, the point \cref{TensorStep} is used in the
composite gradient method~\cite{nesterov2013gradient}. For
$p = 2$, this is a step of composite cubic
Newton~\cite{doikov2018randomized,
	grapiglia2019accelerated}. It can be shown that for $M
\geq p L_p(g)$ the auxiliary optimization problem
in~\cref{TensorStep} is convex for \textit{all} $p \geq
1$ (see Theorem~1 in~\cite{nesterov2018implementable}).
Therefore it can be efficiently solved by different
techniques of convex optimization and linear algebra (see
also~\cite{nesterov2006cubic,nesterov2018implementable}).

Let us mention some properties of point $T \equiv T_M(h;
x)$. Its characteristic condition is as follows:
$$
\ba{rcl}
\left\la
\nabla_{\!y} \Omega_p(g, x; T)
+ {M \over p!}\|T - x\|^{p - 1}B(T - x),
y - T \right\ra + \phi(y)
& \geq & \phi(T), \quad y \in \dom \phi.
\ea
$$
Therefore,
$$
\ba{rcl}
\phi'(T) & \Def & -\nabla_{\!y} \Omega_p(g, x; T) - {M
	\over p!}\|T - x\|^{p - 1}B(T - x) \;\; \in \;\; \partial
\phi(T).
\ea
$$
This inclusion justifies notation $h'(T) \Def \nabla g(T)
+ \phi'(T) \; \in \;
\partial h(T)$.

\newpage

In order to work with these objects, we use the following
result (see Lemma 2 in~\cite{doikov2019local}).

\BL \label{LemmaGradProgress}
Let $\beta \geq 1$ and $M = \beta L_p(g)$. Then
\beq\label{eq-DecFB}
\ba{rcl}
\la h'(T), x - T \ra & \geq & \left( {p! \over (p+1)L_p(g)}
\right)^{1 \over p} \cdot \| h'(T) \|_*^{p+1 \over p}
\cdot {(\beta^2 - 1)^{p-1 \over 2p} \over \beta} \cdot {p
	\over (p^2-1)^{p-1 \over 2p}}.
\ea
\eeq
In particular, if $\beta = p$, then
\beq\label{eq-DecF}
\ba{rcl}
\la h'(T), x - T \ra & \geq & \left( {p! \over (p+1)L_p(g)}
\right)^{1 \over p} \cdot \| h'(T) \|_*^{p+1 \over p}.
\ea
\eeq
\EL

The next lemma describes the global behavior of method
\cref{TensorIterations}.

\BL \label{LemmaGlobal}
Let $\beta \geq 1$ and $M =  \beta L_p(g).$ Then for any
$x, y \in \dom h$ we have
\beq \label{GlobalProgress}
\ba{rcl}
h(T_M(x))
& \leq &
h(y) + {(\beta + 1) L_p(g) \over (p + 1)!}\|y - x\|^{p + 1}.
\ea
\eeq
\EL
\begin{proof}
Indeed,
$$
\ba{rcl}
h(T_M(x)) & = & g(T_M(x)) + \phi(T_M(x)) \\
\\
& \overset{\cref{LipF}}{\leq} &
\Omega_p(g, x; T_M(x))
+ {M \over (p + 1)!}\|T_M(x) - x\|^{p + 1} + \phi(T_M(x)) \\
\\
& \overset{\cref{TensorStep}}{\leq} &
\Omega_p(g, x; y) + {M \over (p + 1)!}\|y - x\|^{p + 1} + \phi(y) \\
\\
& \overset{\cref{LipF}}{\leq} &
g(y) + {M + L_p(g) \over (p + 1)!}\|y - x\|^{p + 1} + \phi(y) \\
\\
& = &
h(y) + {(\beta + 1) L_p(g) \over (p + 1)!}\|y - x\|^{p + 1}.
\ea
$$
\,
\end{proof}

Now, we are ready to prove a convergence result on the
iteration process~\cref{TensorIterations}.

\BT[convergence of tensor method]
Let $M = pL_p(g)$. Then, for every $t \geq 0$ and $y \in
\dom h$ we have
\beq \label{TensorConvergence}
\ba{rll}
\|h'(z_{t + 2})\|_{*}^{p + 1 \over p}
\quad \leq \quad
& \exp\left( - t \cdot \min\left\{
1, \left[
{ p! \, \sigma_{d}(h) 2^{1 - p} \over (p + 1)L_p(g) }
\right]^{1 \over p}
\right\} \cdot \frac{p}{p + 1} \right) \\
\\
& \cdot \left(
(p + 1) L_p(g) \over p!
\right)^{1 \over p}
\cdot
\left(
h(y) - h^{*} + \frac{L_p(g)}{p!}\|y - z_0\|^{p + 1}
\right).
\ea
\eeq
\ET
\begin{proof}
Let us consider the point $z_{t + 1} = T_M(z_t)$.
By~\cref{GlobalProgress}, we have
\beq \label{hOneStep}
\ba{rcl}
h(z_{t + 1}) & \leq & h(y) + {L_p(g) \over p!} \|y - z_t\|^{p + 1}
\ea
\eeq
for any $y \in \dom h$.

\newpage
Denote $x_h^{*} \Def \argmin_{y \in \E}{h(y)}$ and consider
$y = z_t + \alpha(x_h^{*} - z_t)$ for $\alpha \in [0, 1]$.
Then we have
\beq  \label{GlobalAlpha}
\ba{rcl}
h(z_{t + 1}) - h^{*}
& \leq &
h(z_t) - h^{*} - \alpha \left( h(z_t) - h^{*} \right)
+ \alpha^{p + 1} { L_p(g) \over p! }
\|x^{*}_h - z_t\|^{p + 1} \\
\\
& \stackrel{\cref{UnifConv}}{\leq} &
\left(
1 - \alpha + \alpha^{p + 1}
{ (p + 1)L_p(g) \over p! \, \sigma_{d}(h) 2^{1 - p} }
\right)
\cdot \left( h(z_t) - h^{*} \right).
\ea
\eeq
The minimum of the right-hand side is attained at
$$
\ba{rcl}
\alpha^{*} & = & \min\left\{1,
\left[
{  p! \, \sigma_{d}(h) 2^{1 - p} \over (p + 1) L_p(g)   }
\right]^{1 \over p}  \right\}.
\ea
$$
Plugging it into~\cref{GlobalAlpha} gives
\beq \label{OneStepLinear}
\ba{rcl}
h(z_{t + 1}) - h^{*} & \leq &
\left(  1 - \alpha^{*} {p \over p + 1} \right)
\cdot \left( h(z_t) - h^{*} \right) \\
\\
& \leq &
\exp\left( -\alpha^{*} {p \over p + 1} \right)
\cdot (h(z_t) - h^{*}).
\ea
\eeq
Therefore, for every $t \geq 0$ we have
$$
\ba{rcl}
h(z_{t + 1}) - h^{*} & \overset{\cref{OneStepLinear}}{\leq} &
\exp\left( -t \alpha^* {p \over p + 1} \right)
\cdot \left( h(z_1) - h^{*} \right) \\
\\
& \overset{\cref{hOneStep}}{\leq} &
\exp\left( -t \alpha^* {p \over p + 1} \right)
\cdot \left( h(y) - h^{*} + {L_p(g) \over p!}\|y - z_0\|^{p + 1} \right)
\ea
$$
for every $y \in \dom h$. It remains to use~\cref{LemmaGradProgress} 
and finish the proof:
$$
\ba{rcl}
h(z_{t + 1}) - h^{*} & \geq &
h(z_{t + 1}) - h(z_{t + 2}) \\
\\
& \geq &
\la h'(z_{t + 2}), z_{t + 1} - z_{t + 2} \ra \\
\\
& \overset{\eqref{eq-DecF}}{\geq} &
\left(
{ p! \over (p + 1) L_p(g) }
\right)^{1 \over p} \cdot \| h'(z_{t + 2}) \|_{*}^{p + 1 \over p}.
\ea
$$
\,
\end{proof}

Thus, we can see that, applying tensor
method~\cref{TensorIterations} of degree $p \geq 1$ on
step~4 of the general contracting proximal method
(\cref{MainAlgorithm}), we obtain fast linear convergence for the
norms of subgradients. Hence, we can estimate the total
number of inner steps $t_k$ at iteration $k \geq 0$ as
follows.

\begin{corollary}
	Let us minimize function $h_{k + 1}(\cdot)$ by iterations,
	$$
	\ba{rcl}
	z_{t + 1} = T_M(h_{k + 1}; z_t), \quad t \geq 0,
	\ea
	$$
	using $M := pL_p(g_{k + 1})$ and $z_0 := v_k$. Then we have
	$$
	\ba{rcl}
	\| h_{k + 1}'(z_{t_k})\|_{*} & \leq & \delta_{k + 1}
	\ea
	$$
	\newpage
	\noindent
	for
	\beq \label{InnerItersBound}
	\ba{rcl}
	t_k & \geq & 2 +
	\max\left\{
	1, \, {\ell_{k + 1} \over \mu_{k + 1}}
	\right\}
	\cdot {p + 1 \over p}
	\cdot \log \left(
	{ \ell_{k + 1} \, D_{k + 1} \over \delta_{k + 1}^{p + 1 \over p} }
	\right),
	\ea
	\eeq
	where
	\beq \label{EllMuDef}
	\ba{rcccl}
	\ell_{k + 1} & \Def & \left(
	{ (p + 1)L_p(g_{k + 1}) \over p! }
	\right)^{1 \over p},
	\qquad
	\mu_{k + 1} & \Def & \left( \gamma_{k + 1} 2^{1 - p} \right)^{1 \over p},
	\ea
	\eeq
	and
	\beq \label{DkBound}
	\ba{rcl}
	D_{k + 1} & \Def &
	A_{k}(F(x_k) - F^{*}) + \gamma_k \beta_{d}(v_k; x^{*})
	+ \left({ \ell_{k + 1} \over \mu_{k + 1} }\right)^{p}
	\beta_d(v_k; x^{*}) \\
	\\
	& \overset{\cref{InexactGuarantee}}{\leq} &
	R_k(p, \delta)
	\cdot \biggl(1 +  \frac{1}{\gamma_0}
	\left({ \ell_{k + 1} \over \mu_{k + 1} }\right)^{p} \biggr).
	\ea
	\eeq
\end{corollary}
\begin{proof}
By definition, for all $x \in \dom \psi$, we have
$$
\ba{cl}
& h_{k + 1}(x) + A_k \psi(x_k) \\
\\
& = \quad
A_{k + 1}f\bigl({ a_{k + 1}x + A_k x_k \over A_{k + 1}  }\bigr)
+ a_{k + 1}\psi(x) + \gamma_k \beta_d(v_k; x) + A_k \psi(x_k) \\
\\
& \geq \quad
A_{k + 1}F\bigl({ a_{k + 1}x + A_k x_k \over A_{k + 1} }\bigr)
+ \gamma_k \beta_d(v_k; x) \quad \geq \quad A_{k + 1} F^{*}.
\ea
$$
Therefore,
\beq \label{BoundHelper}
\ba{rcl}
-h^{*}_{k + 1} - A_k \psi(x_k) & \leq & - A_{k + 1}F^{*}.
\ea
\eeq
Then for $y \equiv x^{*} \Def \argmin_{y \in \E}F(y)$ we obtain
$$
\ba{rcl}
& & \mspace{-72mu} h_{k + 1}(y) - h_{k + 1}^{*}
+ \frac{L_p(g_{k + 1})}{p!} \|y - z_0\|^{p + 1} 
\, = \, 
h_{k + 1}(x^{*}) - h_{k + 1}^{*}
+ {L_p(g_{k + 1}) \over p!}\|x^{*} - v_k\|^{p + 1} \\
\\
& = &
A_{k + 1}f\bigl({ a_{k + 1}x^{*}
	+ A_k x_k \over A_{k + 1} }\bigr)
+ a_{k + 1} \psi(x^{*})
- h_{k + 1}^{*}
+ \gamma_k \beta_d(v_k; x^{*}) \\
\\
& \; & \quad + \quad  
{L_p(g_{k + 1}) \over p!}\|x^{*} - v_k\|^{p + 1} \\
\\
& \leq & a_{k + 1} F^{*} + A_k F(x_k)
- h_{k + 1}^{*} - A_k \psi(x_k)
+ \gamma_k \beta_d(v_k; x^{*}) \\
\\
& \; & \quad + \quad
{L_p(g_{k + 1}) \over p!}\|x^{*} - v_k\|^{p + 1}\\
\\
& \overset{\cref{BoundHelper}}{\leq} & A_k( F(x_k) - F^{*} )
+ \gamma_k \beta_d(v_k; x^{*})
+ {L_p(g_{k + 1}) \over p!}\|x^{*} - v_k\|^{p + 1}
\; \overset{\cref{ProxUnifConv}}{\leq} \; D_{k + 1}.
\ea
$$
It remains to use this bound together with~\cref{TensorConvergence} and
the following estimation of the strong convexity parameter:
$
\sigma_d(h_{k + 1}) \refGE{GammaKDef}  \gamma_{k + 1}.
$
\end{proof}

By representation~\cref{SubproblemObjective}, we have a
simple relations between Lipschitz constants of the
derivatives for function $g_{k+1}(\cdot)$ and $f(\cdot)$:
\beq \label{LipReg}
\ba{rcl}
L_p(g_{k + 1}) & = & \frac{a_{k + 1}^{p + 1}}{A_{k + 1}^p}
L_p(f), \quad p \geq 1.
\ea
\eeq

\newpage
\noindent
Therefore, we can control the condition number of our
objective. Indeed, by~\cref{InnerItersBound}, the main
complexity factor in the minimization process for $h_{k +
	1}(\cdot)$ is the ratio
$$
\ba{rcl}
{\ell_{k + 1} \over \mu_{k + 1}} & \equiv &
\Bigl({
	(p + 1) L_p(g_{k + 1}) \over p! \,  2^{1 - p} \gamma_{k + 1}
}\Bigr)^{1 \over p}
\; \stackrel{\cref{LipReg},\cref{GammaKDef}}{=} \;
\Bigl({
	(p + 1) 2^{p - 1} a_{k + 1}^{p + 1} L_p(f) \over
	p! \, A_{k + 1}^p (\gamma_0 + A_{k + 1} \sigma_d(\psi))
}\Bigr)^{1 \over p}.
\ea
$$
We are able to keep this ratio small by applying an
appropriate growth strategy for coefficients $A_k$.

Let us consider two cases:
$\sigma_d(\psi) = 0$ and $\sigma_d(\psi) > 0$.
\begin{enumerate}
	
\item  $\sigma_d(\psi) = 0$.
	Let us choose $c
	\equiv
	\frac{p! \, \gamma_0}{2^{p - 1} (p + 1)^{p + 2} L_p(f) }$
	and $a_k \equiv c(p + 1) k^p$.
	Then we have
	$$
	\ba{rcl}
	A_k & = & c(p + 1) \sum\limits_{i = 1}^k i^p \; \geq \;
	c(p + 1) \int\limits_{0}^k x^p dx \; = \; ck^{p + 1},
	\ea
	$$
	and we get
	\beq \label{akFracBound}
	\ba{rcl}
	\frac{a_{k + 1}^{p + 1}}{A_{k + 1}^p} & \leq &
	c (p + 1)^{p + 1}
	\; = \;
	{ p! \, \gamma_0 \over 2^{p - 1} (p + 1) L_p(f) }.
	\ea
	\eeq
	Thus we obtain
	\beq \label{LmuBound}
	\ba{rcl}
	{\ell_{k + 1} \over \mu_{k + 1}} & = &
	\Bigl(
	\frac{a_{k + 1}^{p + 1}}{A_{k + 1}^p} \cdot
	\frac{  2^{p - 1} (p + 1) L_p(f)}{p! \, \gamma_0}
	\Bigr)^{1 \over p}
	\; \stackrel{\cref{akFracBound}}{\leq} \; 1.
	\ea
	\eeq
	
\item  $\sigma_d(\psi) > 0$.
	For $k = 0$ we pick $a_1 \equiv c(p + 1)$
	as in the previous case. Now consider $k \geq 1$.
	Denote
	\beq \label{OmegaDef}
	\ba{rcl}
	\omega & \Def & \min\{ \Bigl(
	\frac{\sigma_d(\psi) p!}
	{L_p(f) (p + 1) 2^{p - 1}}
	\Bigr)^{1 \over p + 1}  , \frac{1}{2} \}
	\ea
	\eeq
	and choose $a_{k + 1}$ from the equation
	$$
	\ba{rcl}
	\frac{a_{k + 1}}{A_{k + 1}} & = &
	\frac{a_{k + 1}}{a_{k + 1} + A_k}
	\; = \; \omega
	\quad \Leftrightarrow \quad
	a_{k + 1} \; = \; \omega (1 - \omega)^{-1} A_k.
	\ea
	$$
	Therefore
	\beq \label{LmuBound2}
	\ba{rcl}
	\frac{\ell_{k + 1}}{\mu_{k + 1}} & \leq &
	\Bigl(
	\frac{a_{k + 1}^{p + 1}}{A_{k + 1}^{p + 1}}
	\cdot \frac{L_p(f) (p + 1) 2^{p - 1}}{p! \, \sigma_d(\psi)}
	\Bigr)^{1 \over p} \\
	\\
	& = & 
	\omega \cdot \Bigl(
	\frac{L_p(f) (p + 1) 2^{p - 1}}{p! \, \sigma_d(\psi)}
	\Bigr)^{1 \over p + 1}
	\; \leq \; 1.
	\ea
	\eeq
	
\end{enumerate}

Thus, in both cases, at every upper-level step we need to
perform a logarithmic number of iterations of the inner
method, multiplied by a small constant.

We are ready to specify the whole optimization procedure.

\vspace*{7pt}
\fbox{
	\hspace*{4pt}
	\centering
	\begin{minipage}{0.85\linewidth}
		\begin{algorithm}[H] 
			\renewcommand{\thealgorithm}{2}
			\caption{\textbf{Contracting Proximal Tensor Method}}
			\label{TensorAlgorithm}
			\noindent\makebox[\linewidth]{\rule{1.087\linewidth}{0.4pt}}
			\begin{algorithmic}[1]
				\vspace{5pt}
				\Require Choose $x_0 \in \dom F$,  inner accuracy $\delta > 0$, $\gamma_0 > 0$.
				\vspace*{5pt}
				\Statex Set $v_0 := x_0$, $A_0 := 0$.
				\vspace*{5pt}
				\Statex Fix
				$\; d(x) := \frac{1}{p + 1}\|x - x_0\|^{p + 1}$,
				\Statex $ \qquad c := \frac{p! \, \gamma_0}{2^{p - 1} (p + 1)^{p + 2} L_p(f)}$,
				$\quad  \omega := \min\{ \Bigl(
				\frac{\sigma_d(\psi) p!}
				{L_p(f) (p + 1) 2^{p - 1}}
				\Bigr)^{1 \over p + 1}  , \frac{1}{2} \} $.
				\vspace*{5pt}
				\Ensure $k \geq 0$. \vspace*{5pt}
				\State \textbf{If} $k = 0$ or $\omega = 0$
				\textbf{Then} \vspace*{5pt}
				\Statex
				$\qquad a_{k + 1} := c(p + 1) (k + 1)^{p}$. \vspace*{5pt}
				\Statex \textbf{Else}
				\Statex \vspace*{5pt}
				$\qquad a_{k + 1} := \omega (1 - \omega)^{-1} A_k$. \vspace*{5pt}
				\State Set $A_{k + 1} := A_k + a_{k + 1}$. \vspace*{5pt}
				\State Denote contracted objective with regularizer: \vspace*{5pt}
				\Statex
				$\qquad g_{k + 1}(x) :=
				A_{k + 1}f\bigl(\frac{a_{k + 1}x + A_k x_k}{A_{k + 1}} \bigr)$,
				\vspace*{5pt}
				\Statex
				$\qquad \phi_{k + 1}(x) := a_{k + 1}\psi(x) + \gamma_k \beta_{d}(v_k; x)$,
				\vspace*{7pt}
				\Statex
				$\qquad h_{k + 1}(x) := g_{k + 1}(x) + \phi_{k + 1}(x)$.
				\vspace*{7pt}
				
				\State Solve inner subproblem by tensor method up to accuracy $\delta$:
				\Statex $\qquad z_0 := v_k$,
				$\; t_{k} := 0$,
				$\; M := p L_p(f) \frac{a_{k + 1}^{p + 1}}{A_{k + 1}^p}$.
				\vspace*{5pt}
				\Statex $\qquad $\textbf{Do} $z_{t_{k} + 1} := T_{M}(h_{k + 1}, z_{t_k}), \; t_{k} := t_{k} + 1$
				\textbf{Until} $\| h'_{k + 1}(z_{t_k})\|_{*} \leq \delta$.
				\vspace*{-5pt}
				\Statex $\qquad v_{k + 1} := z_{t_k}$.
				\vspace*{5pt}
				\State Set $x_{k + 1} := \frac{a_{k + 1}v_{k + 1} + A_k x_k}{A_{k + 1}}$.
				\vspace*{5pt}
				\State Set $\gamma_{k + 1} := \gamma_k + a_{k + 1} \sigma_d(\psi)$.
				\vspace*{5pt}
			\end{algorithmic}
		\end{algorithm}
	\end{minipage}
	\hspace*{4pt}
}
\vspace*{10pt}

Let us present global complexity bounds for this method in
convex and strongly convex cases.

\BT[convex case] \label{TheoremFinal1}
Let for a given $\varepsilon > 0$, the inner accuracy
$\delta$ be fixed as follows:
$$
\ba{rcl}
\delta & = & \Bigl( \frac{p! \, \varepsilon}{L_p(f)}
\Bigr)^{p \over p + 1} \frac{\gamma_0}{2^p (p + 1)^{p +
		1}}.
\ea
$$
Then, in order to achieve $ F(x_K) - F^{*} \leq
\varepsilon $ it is enough to perform
\beq \label{ComplK}
\ba{rcl}
K & = &
\biggl\lfloor
1 +
2^{1 \over p}
\left({
	2^{p - 1} (p + 1)^{p + 2} L_p(f) \, \beta_d(x_0; x^{*}) \over
	\varepsilon \, p!
}\right)^{1 \over p + 1}
\biggr\rfloor
\ea
\eeq
iterations of \cref{TensorAlgorithm}. The total number of oracle
calls $N_K \Def \sum_{k= 1}^K t_k$ is bounded as
\beq \label{ComplTotal}
\ba{rcl}
N_K & \leq & K \cdot \biggl(
3 +
\frac{p + 1}{p} \log\Bigl(
4 \bigl(1 + \frac{1}{\gamma_0}\bigr) ( p + 1 )^{1 \over p}
K^p
\Bigr)
\biggr).
\ea
\eeq
\ET
\begin{proof}
Estimate~\cref{ComplK} follows directly
from~\cref{DeltaChoice1}, by substituting the value 
$$
\ba{rcl}
c & = &
\frac{p! \gamma_0}{2^{p - 1} (p + 1)^{p + 2} L_p(f)}.
\ea
$$
Now, let us prove~\cref{ComplTotal}.
By~\cref{InnerItersBound}, we have
$$
\ba{rcl}
t_k & \leq &
3 + \max\left\{1, \frac{\ell_{k + 1}}{\mu_{k + 1}}  \right\}
\cdot \frac{p + 1}{p} \cdot
\log \left({ \ell_{k + 1} D_{k + 1} \over \delta^{p + 1 \over p} }\right)  \\
\\
& \stackrel{\cref{LmuBound},\cref{DkBound}}{\leq} &
3 + \frac{p + 1}{p} \cdot \log \left({  \gamma_0^{1/p} (1 + \gamma_0^{-1})
	R_k(p, \delta) \over \delta^{p + 1 \over p}
}\right).
\ea
$$
In order to finish the proof, we need to bound the value
under the logarithm.

By the choice of $a_{k}$, we have an upper bound for
$A_k$:
\beq \label{AkUpBound}
\ba{rcl}
A_{k} & = & c(p + 1) \sum\limits_{i = 1}^k i^p
\; \leq \; c(p + 1) \int\limits_{0}^{k + 1} x^p dx
\; = \; c (k + 1)^{p + 1}.
\ea
\eeq
Therefore, for every $0 \leq k \leq K$:
$$
\ba{rcl}
\frac{R_k(p, \delta)}{\delta^{p + 1 \over p}}
& = &
\left(
\frac{(\gamma_0 \beta_d(x_0; x^{*}))^{p \over p + 1}}{\delta}
+ \Bigl(  \frac{(p + 1)2^{p - 1}}{\gamma_0} \Bigr)^{1 \over p + 1} k
\right)^{p + 1 \over p} \\
\\
& \leq &
\left(
\frac{(\gamma_0 \beta_d(x_0; x^{*}))^{p \over p + 1}}{\delta}
+ \Bigl(  \frac{(p + 1)2^{p - 1}}{\gamma_0} \Bigr)^{1 \over p + 1} K
\right)^{p + 1 \over p} \\
\\
& = &
\left(
\Bigl(
\frac{L_p(f) \beta_d(x_0; x^{*})}{p! \, \varepsilon}
\Bigr)^{p \over p + 1}
\frac{2^p (p + 1)^{p + 1}}{\gamma_0^{1 \over p + 1}}
+ \Bigl(  \frac{(p + 1)2^{p - 1}}{\gamma_0} \Bigr)^{1 \over p + 1} K
\right)^{p + 1 \over p} \\
\\
& \overset{\cref{ComplK}}{\leq} &
\left(
\Bigl( \frac{(p + 1)2^{p - 1}}{\gamma_0} \Bigr)^{1 \over p
	+ 1} (K^p + K) \right)^{p + 1 \over p} \; \leq \; \; 4
\Bigl( \frac{p + 1}{\gamma_0} \Bigr)^{1 \over p} K ^p.
\ea
$$
This completes the proof.
\end{proof}

Now, let us discuss the overall dependence of $\delta$ and $K$
on $p$, given by the claim of \cref{TheoremFinal1}.
For simplicity, we fix $\frac{L_p(f)}{\varepsilon} = 1$,
$\beta_d(x_0; x^{*}) = 1$, and $\gamma_0 = 1$. Thus, we observe the functions
\beq \label{f_dep}
\ba{rcl}
\delta(p) & := & \frac{(p!)^{\frac{p}{p + 1}} }{2^p (p + 1)^{p + 1}},
\qquad
K(p) \; := \; 1 + 2^{\frac{1}{p}}
\Bigl(  \frac{2^{p - 1} (p + 1)^{p + 2}}{p!}  \Bigr)^{\frac{1}{p + 1}}.
\ea
\eeq
One can see that $\log_2 \delta(p) \leq -p$.
Therefore, increasing the order of the method,
it requires at least to double the precision of solving the subproblem.
At the same time, we have (using Stirling's formula)
$$
\ba{rcl}
\lim\limits_{p \to +\infty} K(p)
& = &
1 + 2 \exp\Bigl( \lim\limits_{p \to +\infty} \frac{(p + 2) \log (p + 1) - \log p!}{p + 1}  \Bigr)
\;\; = \;\; 1 + 2 \exp(1).
\ea
$$
Hence, the value of $K(p)$ is bounded from above by an absolute constant.
The graphs of the dependence~\eqref{f_dep} are shown in \cref{fig_dependencies}.
Note that in practice, we are interested rather in small values of $p$.

\begin{figure}[ht]
	\begin{center}
		\includegraphics[width=0.35\columnwidth]{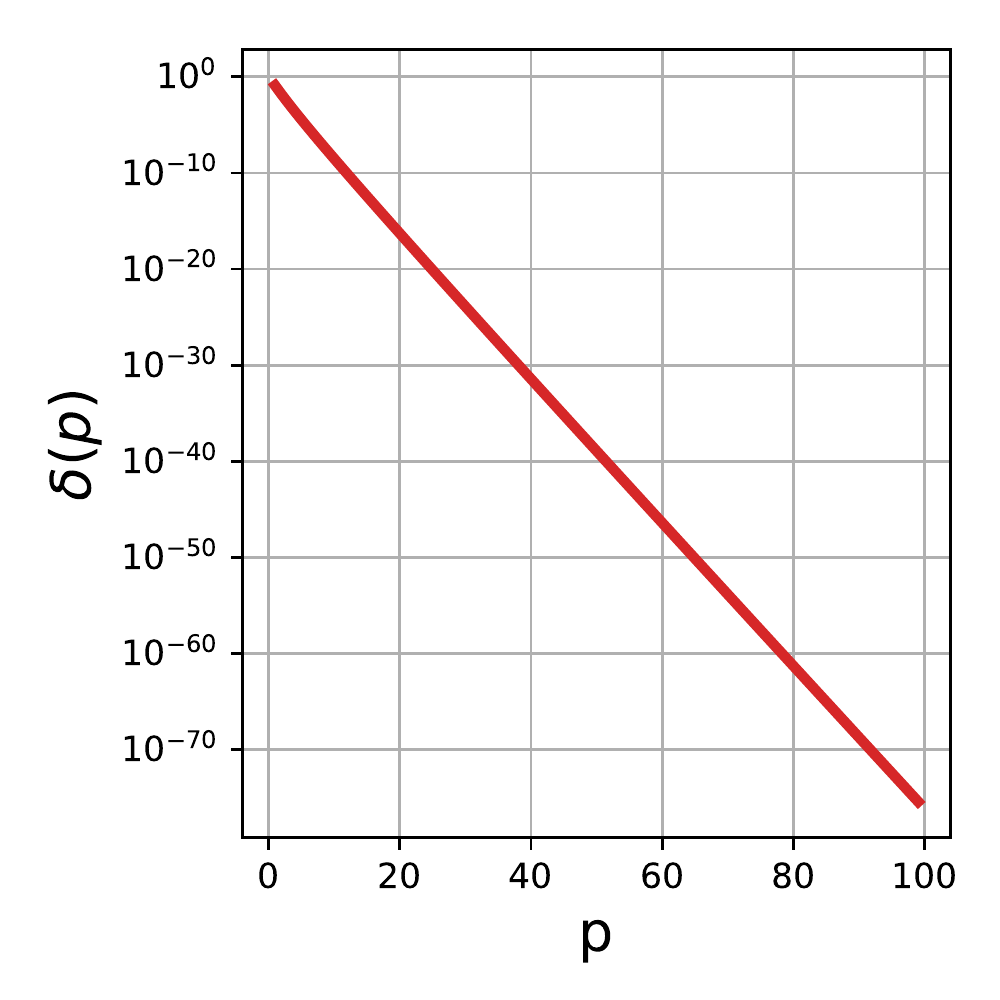}
		\includegraphics[width=0.35\columnwidth]{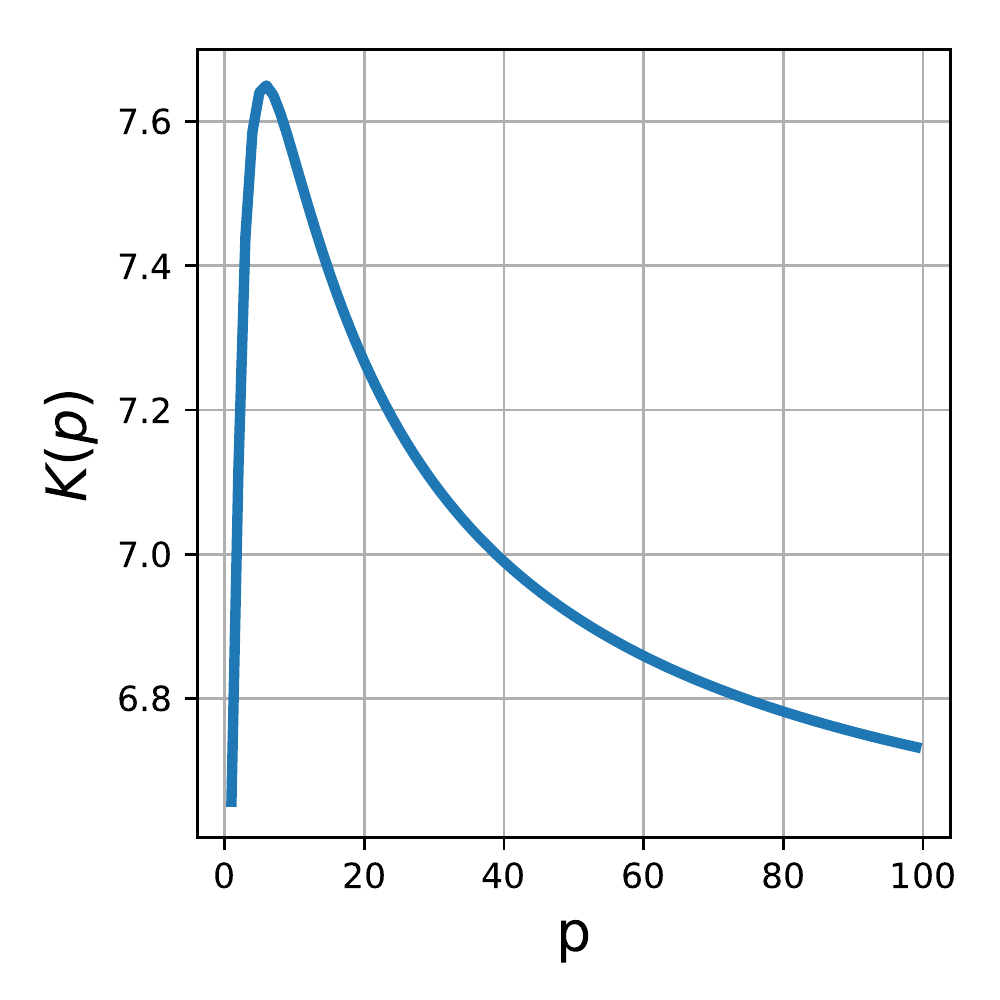}
		\caption{ 
		 The dependence of $\delta$ and $K$ on $p$,
	     while $\frac{L_p(f)}{\varepsilon}$ and $\beta_d(x_0; x^{*})$ are fixed.}
		\label{fig_dependencies}
	\end{center}
	\vskip -0.2in
\end{figure}

\BT[Strongly Convex Case] \label{TheoremFinal2}
Let $\sigma_d(\psi) > 0$ and condition number $\omega$ be
defined as in~\cref{OmegaDef}. Let for a given
$\varepsilon > 0$ the inner accuracy $\delta$ be fixed as
follows:
\beq \label{SConvDelta}
\ba{rcl}
\delta & = & \Bigl( \frac{p! \, \varepsilon}{ L_p(f) }
\Bigr)^{p \over p + 1} \frac{ \gamma_0 p \, \omega } {2^p
	(p + 1)^{ ( (p + 1)^2 + 1 ) / (p + 1) } }.
\ea
\eeq
Then, in order to achieve $ F(x_K) - F^{*} \leq
\varepsilon$, it is enough to perform
\beq \label{TheoremFinal2Iters}
\ba{rcl}
K & = & \left\lfloor
2 + \frac{1}{\omega}
\mathcal{L}
\right\rfloor
\ea
\eeq
iterations of \cref{TensorAlgorithm}, where
$$
\ba{rcl}
\mathcal{L} & \Def & \log\Bigl( \max\bigl\{
\frac{ (p + 1)^p }{\omega^{p + 1}},
\frac{ L_p(f) \beta_d(x_0; x^{*}) (p + 1)^{p + 1} 2^{p + \frac{1}{p}} }
{p! \, \varepsilon}
\bigr\}
\Bigr).
\ea
$$
The total number of oracle calls $N_K$ is bounded as
follows:
\beq \label{TheoremFinal2Oracles}
\ba{rc}
& \; N_K
\; \; \leq \; \;
K \cdot \biggl( 3 +
\bigl( 1 + \frac{e}{(e - 1)p} \bigr)
\cdot \bigl( 1 + \mathcal{L} \bigr) \\
\\
& \quad + \;
\log\Bigl(
\max\{1,
\Bigl(
\frac{4 \sigma_d(\psi) p!}{(p + 1) L_p(f)}
\Bigr)^{1 \over p}  \}
\cdot
\Bigl( 1 + \frac{1}{\gamma_0} \Bigr)
\cdot \frac{(p + 1)^{p + 2 \over p}}{p^{p + 1 \over p}} \cdot
2^{\frac{2p^2 + p + 4}{p}}
\Bigr)
\biggr).
\ea
\eeq
\ET
\begin{proof}
At every iteration $k \geq 1$, we have $A_{k + 1} = (1 -
\omega)^{-1} A_k \geq A_k \exp(\omega) $. At the same
time, we know that
\beq \label{OmegaBound}
\ba{rcl}
\omega & \leq & \frac{1}{2} \; \leq \; \frac{e - 1}{e},
\ea
\eeq
where $e = \exp(1)$. Since for all $\alpha \in [0, 1]$ it holds
that
$$
\ba{rcl}
1 - \frac{e - 1}{e}\alpha & \geq & \exp(-\alpha),
\ea
$$
taking $\alpha = \omega\frac{e}{e - 1} \overset{\cref{OmegaBound}}{\leq} 1$
we obtain
$A_{k + 1}  \leq A_k \exp\bigl(  \omega \frac{e}{e - 1} \bigr)$.
Therefore we have, for all $k \geq 0$,
\beq \label{AkBounds}
\ba{rcl}
A_1 \exp\bigl(k \omega\bigr) & \leq &
A_{k + 1} \; \leq \;
A_1 \exp\Bigl( k \omega  \frac{e}{e - 1} \Bigr).
\ea
\eeq
Now, estimate~\cref{TheoremFinal2Iters} follows directly from~\cref{AkBounds}
and~\cref{KChoice2}
by using the value
$A_1 = \frac{p! \, \gamma_0}{2^{p - 1} (p + 1)^{p + 1} L_p(f)}$.

By the choice of $a_{k + 1}$, we have $\frac{\ell_{k +
		1}}{\mu_{k + 1}} \overset{\cref{LmuBound2}}{\leq} 1$, and
we need only to estimate the value under the logarithm
in~\cref{InnerItersBound}. For every $0 \leq k \leq K$,
we have
$$
\ba{rcl}
\frac{\ell_{k + 1} D_{k + 1}}{\delta^{p + 1 \over p}}
& \stackrel{\cref{LmuBound2},\cref{DkBound}}{\leq} &
\frac{\mu_{k + 1} R_k(p, \delta)
	\left( 1 + \frac{1}{\gamma_0} \right)}{\delta^{p + 1 \over p}} \\
\\
& = &
(\gamma_0 + \sigma_d(\psi)A_{k + 1})^{\frac{1}{p}} 2^{\frac{1}{p} - 1}
\bigr( 1 + \frac{1}{\gamma_0} \bigl) \\
\\
& \; & \quad \cdot \quad
\left(
\frac{(\gamma_0 \beta_d(x_0; x^{*}))^{p \over p + 1}}{\delta}
+ \Bigl(  \frac{(p + 1)2^{p - 1}}{\gamma_0} \Bigr)^{1 \over p + 1} k
\right)^{p + 1 \over p} \\
\\
& \leq &
(\gamma_0 + \sigma_d(\psi)A_{K + 1})^{\frac{1}{p}} 2^{\frac{1}{p} - 1}
\bigr( 1 + \frac{1}{\gamma_0} \bigl) \\
\\
& \; & \quad \cdot \quad
\left(
\frac{(\gamma_0 \beta_d(x_0; x^{*}))^{p \over p + 1}}{\delta}
+ \Bigl(  \frac{(p + 1)2^{p - 1}}{\gamma_0} \Bigr)^{1 \over p + 1} K
\right)^{p + 1 \over p}.
\ea
$$
Let us estimate different terms in this expression
separately.
\begin{enumerate}
	\item By definition of $\omega$, we have
	\beq \label{OmegaBound2}
	\ba{rcl}
	\omega^{p + 1} & \leq & \frac{(p + 1)^p \sigma_d(\psi) A_1}{\gamma_0}.
	\ea
	\eeq
	Therefore,
	$$
	\ba{rcl}
	\gamma_0 + \sigma_d(\psi) A_{K + 1}
	& \overset{\cref{OmegaBound2},\cref{AkBounds}}{\leq} &
	\sigma_d(\psi) A_1 \biggl(
	\frac{(p + 1)^p}{\omega^{p + 1}} + \exp\Bigl( K \omega \frac{e}{e - 1} \Bigr)
	\biggr) \\
	\\
	& \overset{\cref{TheoremFinal2Iters}}{\leq} &
	2 \sigma_d(\psi) A_1 \exp\Bigl( K \omega \frac{e}{e - 1} \Bigr).
	\ea
	$$
	
	\item Substituting the value for $\delta$, we obtain
	$$
	\ba{rcl}
	\frac{( \gamma_0 \beta_d(x_0; x^{*}) )^{p \over p + 1}}{\delta}
	& \overset{\cref{SConvDelta}}{=} &
	\left(
	\frac{L_p(f) \beta_d(x_0; x^{*})}{p! \, \varepsilon}
	\right)^{p \over p + 1}
	\frac{2^p (p + 1)^{((p + 1)^2 + 1) / (p + 1)}}
	{p \, \omega \gamma_0^{1 \over p + 1 }} \\
	\\
	& \overset{\cref{TheoremFinal2Iters}}{\leq} &
	\frac{(p + 1)^2 2^{(2p^2 + p + 1) / (p + 1)} }
	{p \, \omega \gamma_0^{ \frac{1}{p + 1} }}
	\exp \Bigl( K \omega \frac{p}{p + 1} \Bigr).
	\ea
	$$
	
	\item Finally, using that $\exp(x) \geq x$ for all $x \geq 0$, we have
	$$
	\ba{rcl}
	K & \leq &
	\frac{p + 1}{p \, \omega} \exp\Bigl( K \omega \frac{p}{p + 1} \Bigr).
	\ea
	$$
\end{enumerate}
Therefore,
$$
\ba{rcl}
\frac{\ell_{k + 1} D_{k + 1}}{\delta^{p + 1 \over p}} & \leq &
\exp\Bigl(
K \omega  \frac{e}{(e - 1)p}
\Bigr)
\cdot
\Bigl( 2^{2 - p} \sigma_d(\psi) A_1 \Bigr)^{1 \over p}
\cdot
\Bigl( 1 + \frac{1}{\gamma_0} \Bigr) \\
\\
& \; & \cdot \; \biggl(
\frac{
	\exp\bigl( K \omega \frac{p}{p + 1} \bigr)
}{p \, \omega \gamma_0^{1 / (p + 1) }}
\Bigl(
(p + 1)^2 2^{ \frac{2p^2 + p + 1}{p + 1} }
+ (p + 1)^{\frac{p + 2}{p + 1}} 2^{\frac{p - 1}{p + 1}}
\Bigr)
\biggr)^{p + 1 \over p} \\
\\
& < &
\exp\biggl(
K \omega \Bigl( \frac{e}{(e - 1)p} + 1 \Bigr)
\biggr)
\cdot
\Bigl(
\frac{1}{p \, \omega}
\Bigr)^{p + 1 \over p}
\cdot
\Bigl(
\frac{ \sigma_d(\psi) A_1}{\gamma_0}
\Bigr)^{1 \over p} \\
\\
& \; &
\cdot \;
\Bigl( 1 + \frac{1}{\gamma_0} \Bigr)
\cdot (p + 1)^{\frac{2(p + 1)}{p}} 2^{\frac{2p^2 + p + 4}{p}} \\
\\
& = &
\exp\biggl(
K \omega \Bigl( \frac{e}{(e - 1)p} + 1 \Bigr)
\biggr)
\cdot \max\{1,
\Bigl(
\frac{4 \sigma_d(\psi) p!}{(p + 1) L_p(f)}
\Bigr)^{1 \over p}  \} \\
\\
& \; &
\cdot \;
\Bigl( 1 + \frac{1}{\gamma_0} \Bigr)
\cdot \frac{(p + 1)^{p + 2 \over p}}{p^{p + 1 \over p}} \cdot
2^{\frac{2p^2 + p + 4}{p}},
\ea
$$
and we obtain~\cref{TheoremFinal2Oracles}.
\end{proof}

According to~\cref{TheoremFinal1,TheoremFinal2},
the rate of convergence 
for the outer iterations of~\cref{TensorAlgorithm}
is of the same order as the one of the accelerated 
tensor method from~\cite{nesterov2018implementable}.
However, at each step it uses a logarithmic number 
of steps of the basic method. It seems to be 
a reasonable price for the level of generality.
Indeed, we are free to choose
an arbitrary method as the basic one.
The only requirement for it is
the possibility of solving the inner subproblem~\cref{SubproblemSimple}
efficiently.

Note that an additional feature of our methods is
that the sequences of points $\{ x_k \}_{k \geq 0}$ 
and $\{ v_k \}_{k \geq 0}$ form \textit{triangles} (see the rule~\cref{ConvexComb}).
A first-order accelerated method with this nice property 
was discovered in~\cite{gasnikov2018universal}.

\section{Numerical Examples}
\label{SectionNumerical}\SetEQ

\subsection{Quadratic function}

Let us compare numerical performance of 
the contracting proximal method and
the classical proximal point algorithm~\cref{ProxMethod} 
for unconstrained minimization of a convex quadratic function:
$$
\ba{rcl}
f(x) & = & \frac{1}{2}\la Ax, x\ra - \la b, x \ra, \qquad x \in \R^n,
\ea
$$
with $A = A^{*} \succeq 0$. We also run the gradient method and the accelerated
gradient method for this problem. A typical behaviour of the algorithms 
is shown in~\cref{graph1}.
The contracting proximal method has the same
iteration rate as that of the accelerated gradient method,
but requires more gradient evaluations (matrix-vector products)
per iteration.

To compute every step of the proximal algorithms, we use 
the gradient method with line search. We try different 
strategies for choosing inner accuracies $\delta_k$ and end up with a simple 
rule $\delta_k = 1 / k^2$, 
which provides a good balance in the performance of
outer proximal iterations
and the inner method.
(Usually, it requires to do about 4 inner steps per iteration.)

Data was generated randomly, but the set 
of eigenvalues of the matrix was fixed according to the sigmoid function,
for some given $\alpha > 0$
$$
\ba{rcl}
\lambda_i & = & \frac{1}{1 + \exp\bigl( \frac{\alpha}{n - 1}(n + 1 - 2i)  \bigr)},
\qquad 1 \leq i \leq n.
\ea
$$
Therefore it holds that $\lambda_1 = 1 / (1 + \exp(\alpha))$
and $\lambda_n = 1 / (1 + \exp(-\alpha))$, so
parameter $\alpha$ is related to the \textit{condition number} of the problem.

In~\cref{tableQuadratic} we demonstrate the number of iterations
and the total number of matrix-vector products, which are required
for the methods to solve the problem up to $\varepsilon = 10^{-7}$
accuracy in functional residual.

\begin{figure}[ht]
	\begin{center}
		\includegraphics[width=1.0\columnwidth]{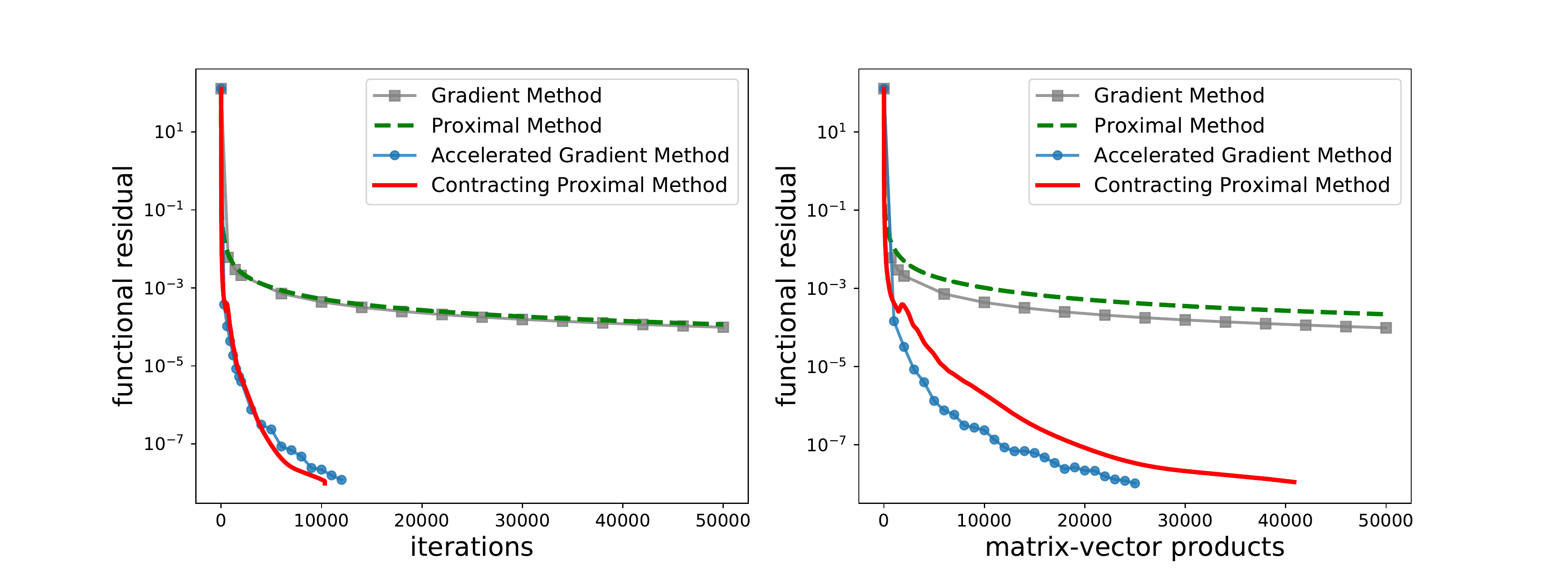}
		\caption{ 
			Convergence of first-order methods on quadratic function.}
		\label{graph1}
	\end{center}
	\vskip -0.2in
\end{figure}

\renewcommand{\arraystretch}{1.5}
\setlength{\tabcolsep}{0.4em}
\begin{table}[h!]
	\scriptsize
	\centering
	\begin{tabular}{ | r|c | p{9ex} | p{9ex}  || p{9ex} | p{9ex}  ||  p{10ex} | p{10ex} || p{10ex} | p{12ex} | p{0ex} }
		\cline{3-10}
		\multicolumn{2}{  r | }{}
		& \multicolumn{2}{c || }{ 
			\begin{tabular}{c} 
				\textbf{Gradient} \\[-1ex] 
				\textbf{Method}  
			\end{tabular}
		}
		&  \multicolumn{2}{ c || }{ 
			\begin{tabular}{c} 
				\textbf{Proximal} \\[-1ex] 
				\textbf{Method}  
			\end{tabular}
		}
		& \multicolumn{2}{ c || }{ 
			\setlength{\tabcolsep}{0ex}
			\begin{tabular}{c} 
				\textbf{Accelerated} \\[-1ex] 
				\textbf{Gradient Method}  
			\end{tabular}
		}
		& \multicolumn{2}{ c | }{ \begin{tabular}{c} \textbf{Contracting} \\[-1ex] 
				\textbf{Proximal Method}  \end{tabular}  
		} \\
		\cline{1-10}
		$n$ & $q$ & \centering iter & \centering mat-vec & \centering iter & \centering mat-vec & \centering iter & \centering mat-vec & \centering iter &  \centering mat-vec &  \\[1ex]
		\cline{1-10}
		500 & $10^{-2}$ & \centering 339 & \centering 339  &  \centering 361 & \centering 1044  &  \centering 115 & \centering 229  &  \centering \textbf{74} & \centering \textbf{137} & \\
		& \centering $10^{-4}$ &     \centering 12158 & \centering 12158  &  \centering 12842 & \centering \centering 36731  & \centering \textbf{350} & \centering \textbf{699}  &  \centering 393 & \centering 1104 & \\
		& $10^{-6}$ &         \centering 96072 & \centering 96072  &  \centering 99269 & \centering 313795  &  \centering \textbf{854} & \centering \textbf{1707}  &  \centering 1081 & \centering 3780 & \\
		\cline{1-10}
		1000 & $10^{-2}$ & \centering 338 & \centering 338  &  \centering 359 & \centering 1035  &  \centering 110 & \centering 219  &  \centering \textbf{73} & \centering \textbf{135} & \\
		& $10^{-4}$ & \centering 11884 & \centering 11884  &  \centering 11912 & \centering 56996  &  \centering \textbf{360} & \centering \textbf{719}  &  \centering 361 & \centering 1014 & \\ 
		& $10^{-6}$ & \centering 77675 & \centering 77675  &  \centering 80758 & \centering 239508  &  \centering \textbf{755} & \centering \textbf{1509}  &  \centering 1117 & \centering 3957 & \\
		\cline{1-10}
	\end{tabular}
	\caption{Minimization of quadratic function, 
		$q = \lambda_{\min}(A) / \lambda_{\max}(A)$.}
	\label{tableQuadratic}
\end{table}

We see that the contracting proximal method
is \textit{always better} than the usual proximal algorithm.
It requires about the same number of iterations as 
the accelerated gradient methods,
but it needs to spend more oracle 
calls per iteration, which confirms the theory.

\subsection{Log-Sum-Exp}

In the next example we compare a performance of
second-order methods for unconstrained minimization 
of the following objective:
$$
\ba{rcl}
f(x) & = & \mu \ln \left( \sum\limits_{i = 1}^m 
\exp\left(  \frac{ \la a_i, x \ra - b_i }{\mu}  \right)  \right),
\qquad x \in \R^n,
\ea
$$
where $\mu > 0$ is a parameter, 
while coefficients of the vectors $\{a_i\}_{i = 1}^m$ and $b$
are randomly generated, and we set $m = 6n$.

We compare the cubically regularized Newton method~\cite{nesterov2006cubic}
and its accelerated variant from~\cite{nesterov2008accelerating}
with the contracting proximal cubic Newton 
(\cref{TensorAlgorithm} with $p = 2$)
for minimizing the objective up to $\varepsilon = 10^{-8}$
accuracy in functional residual. In these algorithms 
we use the following Euclidean norm for the primal space,
$\|x\| = \la Bx, x \ra^{1/2}$,
with matrix $B = \sum_{i = 1}^m a_i a_i^T$, 
and fix the regularization parameter equal to~$1$. 
The results are shown in \cref{tableLogSumExp}.

\renewcommand{\arraystretch}{1.5}
\setlength{\tabcolsep}{0.7em}
\begin{table}[h!]
	\scriptsize
	\centering
	\begin{tabular}{ | r|c | p{9ex} | p{9ex}  || p{9ex} | p{9ex}  ||  p{10ex} | p{14ex} | p{0ex} }
		\cline{3-8}
		\multicolumn{2}{  r | }{}
		& \multicolumn{2}{c || }{ 
			\begin{tabular}{c} 
				\textbf{Cubic Newton}
			\end{tabular}
		}
		&  \multicolumn{2}{ c || }{ 
			\begin{tabular}{c} 
				\textbf{Accelerated} \\[-1ex] 
				\textbf{Cubic Newton}  
			\end{tabular}
		}
		& \multicolumn{2}{ c | }{ 
			\setlength{\tabcolsep}{0ex}
			\begin{tabular}{c} 
				\textbf{Contracting Proximal} \\[-1ex] 
				\textbf{Cubic Newton}  
			\end{tabular}
		}
		\\
		\cline{1-8}
		$n$ & $\mu$ & \centering iter & \centering oracle & \centering iter & \centering oracle & \centering iter & \centering oracle &  \\[1ex]
		\cline{1-8}
		50 & $1$ & \centering 389 & \centering 389  &  \centering 177 & \centering \textbf{353}  &  \centering \textbf{112} & \centering 491  & \\
		& $0.1$ & \centering 482 & \centering 482  &  \centering 202 & \centering \textbf{403}  &  \centering \textbf{141} & \centering 587  &  \\
		& $0.05$ & \centering 886 & \centering 886  &  \centering 343 & \centering \textbf{685}  &  \centering \textbf{236} & \centering 1129  & \\
		\cline{1-8}
		100 & $1$ & \centering 834 & \centering 834  &  \centering 308 & \centering \textbf{615}  &  \centering \textbf{189} & \centering 849  & \\
		& $0.1$ & \centering 1210 & \centering 1210  &  \centering 377 & \centering \textbf{753}  &  \centering \textbf{232} & \centering 1021  &  \\ 
		& $0.05$ & \centering 2598 & \centering 2598  &  \centering 641 & \centering \textbf{1281}  &  \centering \textbf{397} & \centering 1740  &  \\
		\cline{1-8}
	\end{tabular}
	\caption{Comparison of second-order methods on Log-Sum-Exp.}
	\label{tableLogSumExp}
\end{table}

We see that the contracting proximal method outperforms the direct methods in the 
number of iterations, but usually requires additional oracle calls for solving the subproblem.

\section{Conclusion}
\label{SectionConclusion}\SetEQ

In this work, we propose a general acceleration scheme, 
based on proximal iterations.
There are two distinguishing features of our methods: 
employing the \textit{contraction} of the smooth 
component of the objective (this provides the acceleration)
 and flexibility
of \textit{prox-function} 
(its choice should take into account both 
the geometry of the problem and the order of the smoothness).

One of the recent important applications of the accelerated 
proximal point methods in machine learning is the universal framework 
\textit{Catalyst}, applicable to the first-order 
methods~\cite{lin2015universal, lin2018catalyst}.
This is a powerful approach for accelerating many specific optimization 
methods in a common way.
We believe that the results of this paper can help in advancing in this direction, 
resulting in the faster high-order methods for many practical applications.


\end{document}